\documentclass[aap,MSNbibl,seceqn,citesort,dvips]{arximspdf}

\doi{10.1214/11-AAP788}
\volume{22}
\issue{2}
\pubyear{2012}
\firstpage{827}
\lastpage{859}

\makeatletter

\newtheorem{theorem}{Theorem}[section]
\newtheorem{lemma}[theorem]{Lemma}

\newtheorem{proposition}[theorem]{Proposition}
\newproclaim{remark}[theorem]{Remark}
\newproclaim{definition}[theorem]{Definition}

\newtheorem{corollary}[theorem]{Corollary}

\newproclaim{descript}[theorem]{Description}
\newproclaim{assumption}[theorem]{Assumption}

\makeatother

\begin{document}
\begin{frontmatter}

\title{Stochastic model for cell polarity}
\runtitle{Stochastic model for cell polarity}

\begin{aug}
\author[A]{\fnms{Ankit} \snm{Gupta}\thanksref{t1}\corref{}\ead[label=e1]{gupta@math.wisc.edu}}
\runauthor{A. Gupta}
\affiliation{University of Wisconsin, Madison}
\address[A]{Department of Mathematics\\
480 Lincoln Drive\\
Madison, Wisconsin 53706\\
USA\\
\printead{e1}} 
\end{aug}

\thankstext{t1}{Supported in part by NSF Grants DMS-05-53687 and DMS-08-05793.}

\received{\smonth{7} \syear{2010}}
\revised{\smonth{5} \syear{2011}}

%
\begin{abstract}
Cell polarity refers to the spatial asymmetry of molecules on the cell
membrane. Altschuler, Angenent, Wang and Wu have proposed a stochastic
model for studying the emergence of polarity in the presence of
feedback between molecules. We analyze their model further by
representing it as a model of an evolving population with interacting
individuals. Under a suitable scaling of parameters, we show that in
the infinite population limit we get a Fleming--Viot process. Using
well-known results for such processes, we establish that cell polarity
is exhibited by the model and also study its dependence on the
biological parameters of the model.
\end{abstract}

%

\end{frontmatter}

\section{Introduction}\label{sec1}
The phenomenon of polarity is ubiquitous in living organisms. It is
known to occur at many levels: from cellular to organismic. Polarity is
what causes one part of a biological system to be different from
another. Understanding how polarity is established and maintained is a
matter of fundamental concern for biologists.

In this paper we are interested in polarity at the level of individual
cells. Consider a spherical cell consisting of the cytosol and the
membrane. Suppose that it contains numerous molecules that may either
reside in the cytosol or on the membrane. The phenomenon of cell
polarity refers to an identifiable form of spatial asymmetry of
molecules on the membrane. Biologists generally consider a cell to be
in a \textit{polarized} state when most of the membrane molecules appear
to be concentrated around a single site or located in a single
hemisphere on the membrane. It is known that many types of cells
exhibit this phenomenon. The most common example is the yeast cell (see
\cite{BioPB1,BioPB2,AAWWref14,AAWWref15}), but there are many others
(see \mbox{\cite{DN,AAWWref6}}). As noted in \cite{DN}, cell polarity is vital
in the creation of functionally specialized regions on the membrane,\vadjust{\goodbreak}
which can then facilitate cellular processes such as localized membrane
growth, activation of immune response, directional cell migration and
vectorial transport of molecules across cell layers.

Due to its importance, many attempts have been made to investigate the
mechanisms responsible for cell polarity. Drubin and Nelson \cite{DN}
mention that the existence of cell polarity involves positive
feedback
from the signaling molecules on the membrane. This feedback enables the
signaling molecules to perform localized recruitment, thereby causing
concentration of molecules in a specified region on the membrane.
Examples of such signaling molecules include Cdc 42 in budding yeast
(see \cite{AAWWref9}), mPar3/mPar6 in neurons (see~\cite{AAWWref10}),
Rac in kidney cells (see \cite{AAWWref11}) and human chemotaxing
neutrophils (see \cite{AAWWref12}). Even though the feedback mechanism
may bring the molecules together, it is unclear if it can generate cell
polarity alone. This is because the molecules on the membrane are
constantly diffusing and, hence, any clusters that form may disappear
quickly with time. Biologists have proposed that additional mechanisms
like directed transport and coupled inhibitors are required to counter
the spatial diffusion and generate spatial asymmetry (see \cite
{AAWWref3,AAWWref14,AAWWref16,AAWWref4,AAWWref8}). However, these
additional mechanisms are not always found in cells that exhibit
polarity. Hence, the question arises whether feedback alone can cause
polarization.

Altschuler, Angenent, Wang and Wu \cite{AAWW} show that indeed feedback
alone can generate cell polarity when the number of molecules is small.
They prove this result via a simple mathematical model derived by
abstracting the feedback circuits found in cells. In their model, the
feedback mechanism is given by the following: a~molecule on the
membrane may pull a molecule from the cytosol to its location on the
membrane. In a stochastic setting they show that their model exhibits
recurring cell polarity. However, the frequency of polarity is
inversely proportional to the number of molecules in the cell. This
suggests that no polarity can persist in the infinite population limit
without any additional mechanisms to reinforce asymmetry.

In this paper we will scale some parameters of the model in \cite{AAWW}
and study the resulting model. The main result of our paper is that if
we let the feedback strength of each membrane bound molecule increase
linearly with the population size, then we do get recurring cell
polarity in the infinite population limit. Hence, under our scaling,
the model suggests that feedback alone can generate cell polarity in
the infinite population limit without any additional mechanisms. Our
approach is to express the dynamics of cell molecules as a
measure-valued Markov process and then prove that in the limit, the
dynamics of molecules on the membrane can be described by
a~Fleming--Viot process. This process was introduced by Fleming and Viot~\cite{FV}
in 1979 and it has been very well studied since then. An
excellent survey of Fleming--Viot processes is given by Ethier and Kurtz
\cite{EK93}. Using the results already known for such processes, we
will first show that the limiting process is ergodic and hence has a
unique stationary distribution. We will then illustrate that at
stationarity the membrane molecules are arranged into \textit{clans} of
various sizes and molecules in a clan are spatially clustered.
Moreover, the distribution of clan sizes and the expected spatial
spread of the clans can be readily computed in terms of the biological
parameters of the model. Our results will allow us to deduce that there
are times when most of the molecules are part of a single clan and lie
in a single hemisphere on the membrane, thereby causing a cell to
polarize. We now describe the model given in \cite{AAWW}.
\begin{descript}
\label{description}
There are $N$ molecules in the cell (cytosol and membrane). The cell
itself is a sphere of radius $R$. The following four events can change
the molecular configuration in the cell:
\begin{itemize}
\item\textit{Spontaneous membrane association}: A molecule in the cytosol
moves to a~random location on the membrane at rate $k_{\mathrm{on}}$.
\item\textit{Spontaneous membrane dissociation}: A molecule on the
membrane moves back into the cytosol at rate $k_{\mathrm{off}}$.

\item\textit{Membrane association through recruitment} (\textit{feedback
mechanism}): A mole\-cule on the membrane recruits another molecule from
the cytosol at rate
$k_{\mathrm{fb}} \times\mbox{(\textit{fraction of molecules in
the cytosol})}$.

At the time of recruitment, the \textit{recruited} particle moves to the
location of the \textit{recruiting} particle.
\item\textit{Membrane diffusion}: Each molecule on the membrane does
Brownian motion with speed $D$.
\end{itemize}
\end{descript}

The parameters of the model $N$, $D$, $R$, $k_{\mathrm{on}}$,
$k_{\mathrm{fb}}$ and $k_{\mathrm{off}}$ have clear biological
interpretations. As mentioned in \cite{AAWW}, $k_{\mathrm{fb}}$ and
$k_{\mathrm{off}}$ are comparable and throughout this paper we will
assume the following:
\begin{assumption}
\label{assmp1}
\[
k_{\mathrm{fb}} > k_{\mathrm{off}} > 0.
\]
\end{assumption}

In this paper we scale up $k_{\mathrm{fb}}$ and $k_{\mathrm{off}}$ by
the population size $N$ and leave $k_{\mathrm{on}}$ the same. We show
that under this scaling the model becomes mathematically tractable as
$N \to\infty$. In Section \ref{sec3} we will discuss the choice of this scaling
and the necessity of Assumption \ref{assmp1}.

Since we will be relating this model to a well-known model in
population genetics, it is convenient to think of cell molecules as
individuals in an evolving population. Consider the membrane molecules
as being \textit{alive} and the cytosol molecules as being \textit{dead}.
Each membrane molecule has two attributes: location and clan indicator.
When a membrane molecule recruits another molecule from the cytosol,
this new molecule gets initially assigned the same location and clan
indicator as the recruiting molecule. The location of this new molecule
will change subsequently, as it does its own Brownian motion but its
clan indicator remains the same. We can think of membrane recruitment
as a \textit{birth} process in which the recruiting membrane molecule
(the \textit{parent}) passes its characteristics to the recruited
molecule (the \textit{offspring}). The membrane molecules that have the
same clan indicator are said to be in the same \textit{clan}, which
implies that they have a common ancestor. When a molecule spontaneously
associates itself to the membrane, we assign it a new clan indicator
and a randomly chosen location on the membrane. Therefore, we can think
of spontaneous association as \textit{immigration} in which the
individuals bring new genetic traits into the population.~When
a~membrane molecule spontaneously dissociates from the membrane and goes
into the cytosol, it loses both its attributes. So we can think of
spontaneous dissociation as \textit{death}. Note that a molecule that
dies can get reincarnated.

At any time, the membrane molecules can be classified into clans based
on their ancestry. Since the molecules in a clan have a common
ancestor, if the diffusion constant $D$ is small, we can expect them to
be clustered on the membrane. However, the molecular diffusion may
cause a clan to spread apart with time. Surprisingly, this does not
happen in our model. We mentioned before that in the infinite
population limit, the cell dynamics is ergodic and reaches a stationary
state at which the spatial spread of the clans does not change with
time. This is due to the extremely fast nature of the birth and death
mechanisms in our model which causes most of the molecules in a clan to
be \textit{newly born}. Hence, they have been unable to move away from
their common ancestor by too much. We will show that in the limit there
are infinitely many clans present in the population at stationarity,
but there are only a few \textit{large} clans. These two results together
imply that spatial asymmetry is present and persistent. We will then
argue that there will be times when most of the population will be part
of one large clan and also appear to concentrate around a single point.
Consequently, the cell is polarized at these times. This shows that
unlike the original model, cell polarity is present in our rescaled
model as the population size goes to infinity. For a detailed study of
the model considered here, we refer the readers to \cite{AnkitThesis}.

This paper is organized as follows. In Section \ref{sec2} we give the main
results of our paper. In Section \ref{sec3} we interpret these results
in the
context of biology and compare our results with the results provided in
\cite{AAWW} for the original model. We also state some interesting
research questions that we were unable to answer in this paper.
Finally, in Section \ref{sec4} we provide the proofs of our results.

\section*{Notation}
We now introduce some notation that will be used throughout the paper.
Let $(S,d)$ be a compact metric space. Then by $B(S) ( C(S)
) $ we refer to the set of all bounded (continuous) real-valued Borel
measurable functions. Since $(S,d)$ is compact, $C(S) \subset B(S)$.
Both $B(S)$ and $C(S)$ are Banach spaces\vadjust{\goodbreak} under the sup norm $\| f
\|= \sup_{x \in S } | f(x) |$. For any differentiable
manifold $M$ and $k \geq1$, let $C^{k}(M)$ be the space of functions
which are $k$-times continuously differentiable. Let $\mathcal{B}(S)$
be the Borel sigma field on $S$. The space of all positive Borel
measures with total measure bounded above by $1$ is denoted by $\mathcal
{M}_1(S)$ and $\mathcal{P}(S)$ is the space of all Borel probability
measures. Since $(S,d)$ is compact, Prohorov's theorem implies that
both~$\mathcal{P}(S)$ and $\mathcal{M}_1(S)$ are compact under the
topology of weak convergence. For any $\mu\in\mathcal{M}_1(S)$ and
$f\dvtx S \to\mathbb{R}$ let
\[
\langle f,\mu\rangle=\int_{S} f(s)\mu(ds).
\]
If $\mu\in\mathcal{M}_1(S)$, then for any positive integer $m$, $\mu
^{m} \in\mathcal{M}_1(S^m)$ refers to the $m$-fold product of $\mu$.
If $\mu$ is an atomic measure of the form $a_n \sum_{i=1}^{n} \delta
_{x_i} $ for some $a_n>0$, then $\mu^{(m)}$ is the \textit{symmetric}
$m$-fold product of $\mu$ defined by
%
%
\begin{equation}
\label{mum}
\mu^{(m)} = \frac{1}{n(n-1)\cdots(n-m+1)} \sum_{1 \leq
i_1 \neq i_2\neq\cdots\neq i_m \leq n }
\delta_{(x_{i_1},x_{i_2},\ldots,
x_{i_m} )},
\end{equation}
where the sum is over all distinct $m$-tuples of $\{1,2,\ldots,n\}$. If
$m>n$, then the sum above is empty and $\mu^{(m)}$ is taken to be $0$.
Observe that $\mu^{(m)}$ does not depend on $a_n$ and for $n>m$ it is a
probability measure over $S^m$. Also note that if $\mu$ is a
probability measure (i.e., $a_n = 1/n$), then for large $n$, $\mu
^{(m)}$ is approximately equal to $\mu^{m}$.

The space of cadlag functions (i.e., right continuous functions with
left limits) from $[0,\infty)$ to $S$ is called $D_{S} [0,\infty)$ and
it is endowed with the Skorohod topology (for details see Chapter 3,
Ethier and Kurtz \cite{EK}). The space of continuous functions from
$[0,\infty)$ to $S$ is called $C_{S} [0,\infty)$ and it is endowed with
the topology of uniform convergence over compact sets.

For any operator $A \subset B(S) \times B(S), $ let $\mathcal{D }(A)$
and $\mathcal{R}(A)$ designate the domain and range of $A$. The notion
of the \textit{martingale problem} associated to an operator $A$ is
introduced and developed in Chapter 4, Ethier and Kurtz~\cite{EK}. In
this paper, by a solution of the martingale problem for $A$, we mean a
measurable stochastic process $X$ with paths in $D_{S} [0,\infty)$ such
that for any $f \in\mathcal{D }(A)$,
\[
f(X(t)) - \int_{0}^{t} A f(X(s))\,ds
\]
is a martingale with respect to the filtration generated by $X$. For a
given initial distribution $\pi\in\mathcal{P}(S)$, a solution $X$ of
the martingale problem for $A$ is a solution of the martingale problem
for $(A,\pi)$ if $\pi= P X(0)^{-1}$. If such a solution exists
uniquely for all $\pi\in\mathcal{P}(S)$, then we say that the
martingale problem for $A$ is well posed.

\section{The main results}\label{sec2}

Our first task in this section is to represent the dynamics of cell
molecules as a measure-valued Markov process. Suppose there are $N$
molecules in the\vadjust{\goodbreak} cell (cytosol and membrane). The cell membrane will be
denoted by $E$ and it is a sphere of radius $R$ in $\mathbb{R}^3$. As
we mentioned before, each membrane molecule has two attributes:
location and clan indicator. The locations are elements in $E$, while
the clan indicators will be chosen as elements in the unit interval
$[0,1]$. Hence, $E \times[0,1]$ is the type space for the molecules. A
molecule of type $x=(y,z) \in E \times[0,1]$ is located at $y$ on the
membrane and has $z$ as its clan indicator. Note that a membrane
molecule will change its type only due to Brownian motion. Therefore,
during its stay on the membrane, only its location (first coordinate)
changes while its clan indicator (second coordinate) remains
fixed.

If there are $N$ molecules in the cell, then we assign mass $1/N$ to
each molecule. The membrane population at time $t$ can be represented
by an atomic measure as follows:
%
%
\begin{equation}
\label{defmun}
\mu^N(t) = \frac{1}{N} \sum_{i=1}^{n^N(t)} \delta_{x_i(t)},
\end{equation}
where $n^N(t) = N \langle1,\mu^N(t)\rangle$ is the number
of molecules on the membrane at time $t$ and $x_1(t),\ldots,
x_{n^N(t)}$ are their types. Viewed as a process, $\mu^N$ is Markov and
its state space is given by
\begin{eqnarray*}
&&\mathcal{M}^N_{a}(E \times[0,1])\\
&&\qquad=\Biggl\{ \frac{1}{N} \sum_{i=1}^n
\delta_{x_i} \dvtx0 \leq n \leq N \mbox{ and } x_1,\ldots,x_n \in
E \times[0,1] \Biggr\}.
\end{eqnarray*}
For any $\mu\in\mathcal{M}^N_{a}(E \times[0,1])$, the total mass
$\langle1,\mu\rangle\leq1$ and, hence, $\mathcal
{M}^N_{a}(E \times[0,1]) \subset\mathcal{M}_1(E \times
[0,1])$. If we
endow $\mathcal{M}^N_{a}(E \times[0,1])$ with the topology of weak
convergence, then it is a compact space.

The generator of any Markov process is an operator which captures the
rate of change of the distribution of the process. For a detailed
discussion on generators, see Chapter 4 in Ethier and Kurtz \cite{EK}.
For a speed $D$ Brownian motion on the membrane $E$, the generator is
given by $\frac{D}{2} \Delta$, where $\Delta$ is the Laplace--Beltrami
operator on $E$. Note that $C^2(E) \subset\mathcal{D}(\Delta)$, where
$C^2(E)$ is the space of twice continuously differentiable functions on
$E$. Next we define the classes of functions that we will use in this paper.
\begin{definition}
\label{defc2}
\begin{eqnarray*}
\mathcal{C} & = & \bigl\{f \in C \bigl(( E \times[0,1] )^{m}
\bigr) \mbox{ such that } f(\cdot,z) \in C^2(E^{m}) \\
&&\hspace*{3pt}\mbox{ for all } z\in[0,1]^{m}\mbox{,
and } \nabla f (x,\cdot), \Delta f(x,\cdot) \in
C([0,1]^m) \\
&&\hspace*{133pt}\mbox{ for all } x\in E^m \mbox{ and
} m \geq1 \bigr\}.
\end{eqnarray*}
\end{definition}

\begin{definition}
\label{defc02}
\[
\bar{\mathcal{C}}=\bigl\{F(\mu)=\bigl\langle f, \mu^{(m)}
\bigr\rangle\mbox{ such that } f \in\mathcal{C} \mbox{ and }
m \geq1 \bigr\}.
\]
\end{definition}

We now specify the generator $\mathbb{A}^N$ for the process $\mu^N$,
which captures the dynamics of the cell population. For any $f \in
\mathcal{C}$ let $\Delta_i f$ denote the action of the Laplace--Beltrami
operator on $f$ by considering it as a function of its $i$th
coordinate. Let the domain of the operator $\mathbb{A}^N$ be $\mathcal
{D}(\mathbb{A}^N)= \bar{\mathcal{C}}$ and for $F \in\bar{\mathcal{C}}$
of the form $\langle f, \mu^{(m)}\rangle$, define
%
%
\begin{eqnarray}
\label{genpm1}\quad
\mathbb{A}^N F(\mu) & = & \frac{D}{2} \sum_{i=1}^{m} \bigl\langle\Delta_i
f,\mu^{(m)} \bigr\rangle\nonumber\\
&&{} + k_{\mathrm{on}}N (1-h) \int_{E} \int
_{[0,1]} \biggl(F\biggl(\mu+ \frac{1}{N} \delta_{(y,z)} \biggr)-F(\mu)
\biggr) \vartheta(dy) \,dz \nonumber\\[-8pt]\\[-8pt]
&&{} + k_{\mathrm{off}} N^2 \int_{E \times[0,1]} \biggl(F\biggl(\mu- \frac
{1}{N} \delta_{x} \biggr)-F(\mu)\biggr)\mu(dx) \nonumber\\
&&{} + k_{\mathrm{fb}} N^2(1-h) \int_{E \times[0,1]} \biggl( F\biggl(\mu+
\frac{1}{N} \delta_{x} \biggr)-F(\mu) \biggr)\mu(dx), \nonumber
\end{eqnarray}
where $h=\langle1,\mu\rangle$ and $\vartheta(\cdot)$ is the surface
area measure on the sphere $E$ normalized to have total area as $1$.
Terms in the operator above correspond to the surface diffusion of the
membrane molecules, spontaneous association, spontaneous dissociation
and membrane recruitment, in that order. The martingale problem for
$\mathbb{A}^N$ is well posed and this can be seen by viewing the
operator $\mathbb{A}^N$ as a bounded perturbation of the diffusion
operator (given by the first term of $\mathbb{A}^N$). It is easy to
argue that the martingale problem for the diffusion operator is well
posed and the solution for any initial distribution is just the
empirical measure process of a system of particles doing independent
speed $D$ Brownian motion on $E$. Proposition 10.2 and Theorem~10.3 in
Chapter 4 of Ethier and Kurtz \cite{EK} imply the well-posedness of the
martingale problem for $\mathbb{A}^N$.

It will soon become evident that the initial distribution of cell
molecules is not important for the discussion in this paper. For
definiteness we will assume that the membrane is initially empty. Let
$\bar{\pi}_0 \in\mathcal{P}(\mathcal{M}^N_{a}(E \times
[0,1]))$ be the distribution that puts all the mass at the $0$
measure. From now on $\mu^N$ will be the unique Markovian solution to
the martingale problem corresponding to $(\mathbb{A}^N,\bar{\pi
}_0)$.

We define another process $h^N$ by
%
%
\begin{equation}
\label{defhn}
h^N(t) = \langle1,\mu^N(t) \rangle= \frac{n^N(t)}{N} ,\qquad
t\geq0.
\end{equation}
At any time $t$, $h^N(t)$ is the fraction (or the total mass) of cell
molecules that are on the membrane. We will refer to $h^N$ as the
\textit
{fraction process}. Observe that $h^N(0) = 0$.

We are interested in showing the convergence of the sequence of
processes~$\{ \mu^N\}$\vadjust{\goodbreak} as $N \to\infty$. Note that the last two terms
in $\mathbb{A}^N$ [see (\ref{genpm1})] do not appear to converge
independently. This is because terms like
\[
\int_{E \times[0,1]} \biggl(F\biggl(\mu\pm\frac{1}{N} \delta_{x}
\biggr)-F(\mu)\biggr)\mu(dx)
\]
will typically be of order $1/N$ and we are multiplying them by $N^2$
outside. However, convergence does happen because these two terms
combine to give a second-order term. Instead of directly dealing with
the sequence of generators $\{\mathbb{A}^N\}$, we will prove the
convergence result in Section \ref{sec4} by using the particle construction
introduced by Donnelly and Kurtz in \cite{DK99}. This construction
provides a more probabilistic way of passing to the limit.

Define an operator $\mathbb{A}$ as follows. For any $F \in C (
\mathcal{M}_1 (E \times[0,1]) )$ of the form $F(\mu)=
\prod_{i=1}^{m} \langle f_i,\mu\rangle$, where $f_i \in
\mathcal{C} \cap C(E \times[0,1] )$ for all $i=1,2,\ldots
,m$ and $m \geq1$, define
%
%
\begin{eqnarray}
\label{gena}
\mathbb{A} F(\mu)& = &\frac{D}{2} \sum_{i=1}^{m} \langle\Delta f_i,\mu
\rangle\prod_{j \neq i} \langle f_j,\mu\rangle\nonumber\\[-3pt]
&&{}+
k_{\mathrm{on}} \frac{(1-h_{\mathrm{eq}})}{h_{\mathrm{eq}}}
\nonumber\\[-9.5pt]\\[-9.5pt]
&&\hspace*{11pt}{}\times\sum
_{i=1}^{m} \int_{E} \int_{[0,1]} \bigl( f_i(y,z) \vartheta(dy) \,dz -
f_i(x) \mu(dx) \bigr) \prod_{j \neq i} \langle f_j ,\mu\rangle
\nonumber\\[-3pt]
&&{}+\frac{k_{\mathrm{fb}}(1-h_{\mathrm{eq}})}{h_{\mathrm{eq}}}
\sum_{1\leq i \neq j \leq m} ( \langle f_i f_j,\mu
\rangle- \langle f_i,\mu\rangle\langle f_j,\mu
\rangle) \prod_{k \neq i,j} \langle f_k,\mu
\rangle. \nonumber
\end{eqnarray}
The operator $\mathbb{A}$ is the generator of a Fleming--Viot process.
The martingale problem corresponding to it is well posed and each
solution has paths in $C_{\mathcal{P}(E \times[0,1])}[0,\infty)$ by
Theorem 3.2, Ethier and Kurtz \cite{EK93}. We are now ready to state
the convergence result. Throughout this paper $\Rightarrow$ will denote
convergence in distribution.
\begin{theorem}
\label{mainconvthm}
There exists a stopping time $\tau_N$ (with respect to filtration
generated by $\mu^N$) satisfying $\tau_N \to0$ a.s. as $N \to\infty$,
such that if we define processes $\hat{h}^N$ and $\hat{\mu}^N$ as
%
%
\begin{equation}
\hat{h}^N (t) = h^N(t+\tau_N),\qquad t \geq0,
\end{equation}
and
%
%
\begin{equation}
\label{hatmu}
\hat{\mu}^N (t) = \mu^N(t+\tau_N),\qquad t \geq0,
\end{equation}
then the following is true.
\begin{longlist}[(A)]
\item[(A)] For any $T>0$,
\[
\sup_{0 \leq t\leq T} |\hat{h}^N(t)-h_{\mathrm{eq}} |
\Rightarrow0 \qquad\mbox{as } N \to\infty,
\]
where $h_{\mathrm{eq}} = 1 - \frac{k_{\mathrm{off}}}{k_{\mathrm{fb}}}$.\vadjust{\goodbreak}
\item[(B)] Suppose that the sequence of random variables $\{\hat{\mu
}^N(0)\}$ converges in distribution to $\mu(0)$ and let $\pi_0 \in
\mathcal{P} ( \mathcal{P}(E \times[0,1]) )$ be the
distribution of $\mu(0)/h_{\mathrm{eq}}$. Then $\hat{\mu}^N \Rightarrow
\mu$ in $D_{\mathcal{M}_1 (E \times[0,1])}[0,\infty)$ as $N \to\infty
$, where $\mu= h_{\mathrm{eq}} \nu$ and $\nu$ is the Fleming--Viot
process with type space $E \times[0,1]$, generator $\mathbb{A}$ and
initial distribution $\pi_0$.\vspace*{3pt}
\end{longlist}
\end{theorem}

\begin{remark}
\label{rem2}
Note that the state space of the processes $\hat{\mu}^N$ is $\mathcal
{M}_1(E \times[0,1])$, which is compact and so $\mathcal{P} (
\mathcal{M}_1(E \times[0,1]) )$ is also compact by Prohorov's
theorem. Hence, the distributions of $\hat{\mu}^N(0)$ will certainly
converge along a subsequence and the assertion of the theorem above
will hold for this subsequence. In fact, the distributions of $\hat{\mu
}^N(0)$ converge along the entire sequence (see Remark~\ref
{convergenceofinitialdistribution}).\vspace*{3pt}
\end{remark}

A heuristic explanation for the above result is as follows. As $N$ gets
larger, the extremely fast nature of the birth and death mechanisms
forces the fraction process to immediately settle to an equilibrium
value $h_{\mathrm{eq}}$ [given by part (A) of the above theorem]. Note
that $k_{\mathrm{fb}}h_{\mathrm{eq}}(1-h_{\mathrm{eq}}) = k_{\mathrm
{off}} h_{\mathrm{eq}}$ and so at this equilibrium value, the net
influx of population onto the membrane due to birth matches the net
efflux of population from the membrane due to death. Since the total
mass on the membrane is not allowed to deviate from this equilibrium,
any addition of new mass due to immigration must be concurrently offset
by an equal reduction in existing mass due to death. Hence, in the
limit, the net demographic effect of immigration is the same as that of
mutation and, therefore, we see a mutation-like term in the limiting
generator $\mathbb{A}$ [see the second term in (\ref{gena})].
Similarly, the addition of new mass on the membrane due to birth must
be accompanied by the reduction of equal mass due to death. This gives
rise to the second-order sampling term in $\mathbb{A}$ [see the third
term in (\ref{gena})]. These ideas are made rigorous in the proof of
Theorem \ref{mainconvthm} given in Section \ref{sec4}.

From now on $\nu$ will denote the Fleming--Viot process given in the
statement of Theorem \ref{mainconvthm}. We next claim that $\nu$ has a
unique stationary distribution and it is also \textit{strongly ergodic}
in the sense that its transition function converges asymptotically to
the stationary distribution. In fact, this convergence is exponentially
fast. Let $S$ be any metric space and let $\mathcal{B}(S)$ be the Borel
sigma field on $S$. Define the \textit{total variation} metric over the
space of probability measures $\mathcal{P}(S)$ by
\[
\|v_1-v_2 \|_{\mathrm{var}}=\sup_{\Gamma\in\mathcal
{B}(S)} \|v_1(\Gamma)-v_2(\Gamma) \|.
\]

\begin{proposition}
\label{stationarity}
\textup{(A)} The process $\nu$ is strongly ergodic and it has a~unique
stationary distribution $\Pi\in\mathcal{P}(\mathcal{P}(E
\times[0,1]))$.

\textup{(B)} The transition function of $\nu$ converges to the stationary
distribution exponentially fast. There exists a constant $C>0$ such that
\[
\bigl\| P\bigl(\nu(t) \in\cdot\bigr) -\Pi(\cdot) \bigr\|_{\mathrm{var}} \leq C \exp{
\biggl(-\biggl(\frac{k_{\mathrm{on}} (1-h_{\mathrm{eq}})}{2 h_{\mathrm
{eq}}}\biggr) t \biggr)} .
\]
\end{proposition}
\begin{pf}
Both parts follow from Theorem 5.1 and Corollary 8.4 in Ethier and
Kurtz \cite{EK93}.
\end{pf}

Define a $\mathcal{P}([0,1])$-valued process $\nu_c$ by
%
%
\begin{equation}
\label{clanprocess}
\nu_c(t,A) = \nu(t, E \times A),\qquad A \in\mathcal{B}([0,1]) \mbox{
and } t \geq0.
\end{equation}
We will refer to $\nu_c$ as the \textit{clan process}, as it will help us
in the determination of clan sizes. We shall discuss this further in
Section \ref{sec3}. As a consequence of Theorem~\ref{mainconvthm}, we
get the
following corollary.
\begin{corollary}
\label{corrclanprocess}
The process $\nu_c$ is the Fleming--Viot process with type space $[0,1]$
and generator $\mathbb{A}_c$ given by
\begin{eqnarray*}
\mathbb{A}_c F(\mu)& = &k_{\mathrm{on}} \frac{(1-h_{\mathrm
{eq}})}{h_{\mathrm{eq}}} \sum_{i=1}^{m} \int_{[0,1]} \bigl( f_i(z) \,dz -
f_i(x) \mu(dx) \bigr) \prod_{j \neq i} \langle f_j ,\mu\rangle
\\
&&{}+\frac{k_{\mathrm{fb}}(1-h_{\mathrm{eq}})}{h_{\mathrm{eq}}} \sum
_{1\leq i \neq j \leq m} ( \langle f_i f_j,\mu\rangle
- \langle f_i,\mu\rangle\langle f_j,\mu
\rangle) \prod_{k \neq i,j} \langle f_k,\mu\rangle,
\end{eqnarray*}
where $F(\mu) = \prod_{i=1}^{m} \langle f_i , \mu\rangle$
and $f_i \in C([0,1])$ for $i=1,2,\ldots,m$.
\end{corollary}
\begin{pf}
The proof is immediate from the definition of $\nu_c$ and the
descriptions of the generators $\mathbb{A}$ and $\mathbb{A}_c$.
\end{pf}

Since the molecules are constantly diffusing on the membrane, we would
expect each clan to spread out more and more with time. However, we
will show that this does not happen in our model. We would like to
measure the average spatial spread of the molecules that belong to the
same clan. One way to measure it would be to randomly sample two
molecules from the membrane population at any time $t$ and compute
their expected distance squared, given that they are in the same clan.
We call this quantity $S_p(t)$. For $i=1,2$ let $X_i(t) =
(Y_i(t),C_i(t)) \in E \times[0,1]$ be the sampled molecules. Then
given $\nu(t)$, $X_1(t)$ and $X_2(t)$ are i.i.d. with common
distribution $\nu(t)$.
Therefore,
\begin{eqnarray*}
S_p(t)& = &E \bigl( \| Y_1(t)-Y_2(t) \|^2 \vert
C_1(t)=C_2(t)\bigr) \\
& = &\frac{E ( \| Y_1(t)-Y_2(t) \|^2 1_{ \{
C_1(t)=C_2(t) \}})}{P ( C_1(t)=C_2(t) )} \\
& = &\frac{E ( E ( \| Y_1(t)-Y_2(t) \|^2 1_{ \{
C_1(t)=C_2(t) \}} \vert\nu(t) ) )}{ E ( E
( C_1(t)=C_2(t) \vert\nu(t) ) )} \\
&=&\frac{E (\int_{E} \int_{[0,1]} \| y_1-y_2 \|^2 1_{\{
c_1=c_2\}} \nu(t,dy_1,dc_1)\nu(t,dy_2,dc_2))}{E(\int_{E} \int
_{[0,1]} 1_{\{c_1=c_2\}} \nu(t,dy_1,dc_1)\nu(t,dy_2,dc_2))}.
\end{eqnarray*}
From Proposition \ref{stationarity} we know that the process $\nu$ has
a unique stationary distribution $\Pi\in\mathcal{P}(\mathcal{P}(E
\times[0,1]))$. At stationarity, $S_p(t)$ does not depend on~$t$ and
can be written as
%
%
\begin{eqnarray}
\label{spreadst}
S_p&=& \int_{\mathcal{P}(E \times[0,1])}
\biggl(\int_{E} \int_{[0,1]} \| y_1-y_2 \|^2 1_{\{c_1=c_2\}}
\mu(dy_1,dc_1)\mu(dy_2,dc_2)\biggr) \nonumber\\
&&\hphantom{\int_{\mathcal{P}(E \times[0,1])}}
{}\times\Pi(d \mu) \\
&&{}\times\biggl({\int_{\mathcal{P}(E
\times[0,1])} \biggl(\int_{E} \int_{[0,1]} 1_{\{c_1=c_2\}}
\mu(dy_1,dc_1)\mu(dy_2,dc_2)\biggr)\Pi(d \mu)}\biggr)^{-1}.\hspace*{-18pt}\nonumber
\end{eqnarray}
The theorem below gives a precise formula for $S_p$. It will be proved
in Section~\ref{sec4}.

\begin{theorem}
\label{theoremspread}
Let $\alpha=\frac{1-h_{\mathrm{eq}}}{h_{\mathrm{\mathrm{eq}}}}=\frac
{k_{\mathrm{off}}}{k_{\mathrm{fb}}-k_{\mathrm{off}}}$. Then
\[
S_p=\frac{2 D}{ ( (k_{\mathrm{on}}+k_{\mathrm{fb}}) \alpha+
{D}/{R^2} )}.
\]
\end{theorem}

In the next section we connect all the results mentioned in this
section and present the complete picture in our biological setting.

\section{The biological interpretation}\label{sec3}

In this paper our main objective is to show that if we take the model
for cell polarity given by Altschuler, Angenent, Wang and Wu \cite
{AAWW} (see Description \ref{description}) and scale the parameters~$k_{\mathrm{fb}}$
and $k_{\mathrm{off}}$ by the population size $N$,
then, unlike the original model, we get cell polarity in the infinite
population limit. In this section we describe how the results mentioned
in the last section help us in making this conclusion. These results
will also give us an insight into the influence of various biological
parameters on polarity.

The main convergence result, Theorem \ref{mainconvthm}, shows that as
$N \to\infty$ the fraction of the molecules on the membrane at any
time is equal to $h_{\mathrm{eq}}$ and the dynamics of cell molecules
is given by a measure-valued process $\mu$ where $\mu= h_{\mathrm{eq}}
\nu$ with $\nu$ being a Fleming--Viot process. The process $\nu$ has a
unique stationary distribution and its transition function converges
exponentially to this stationary distribution (see Proposition \ref
{stationarity}).

At any time, the molecules on the membrane can be divided into clans
based on their ancestral relationships. We now determine the
distribution of the clan sizes. Let $\nu_c$ be the process given by
(\ref{clanprocess}). From Corollary \ref{corrclanprocess} we know that
it is a Fleming--Viot process with type space $[0,1]$ and generator~$\mathbb{A}_c$.
Such a Fleming--Viot process arises as a reformulation
of the infinitely-many-neutral-alleles model due to Kimura and Crow
\cite{KC64} (see \cite{EK} for more details). By Theorem 7.2 in Ethier
and Kurtz \cite{EK93}, at any time $t$, the random probability measure
$\nu_c(t)$ is purely atomic. This means that $\nu_c(t)$ is of the form
$\sum_{i=0}^{\infty} p_i \delta_{x_i}$, where $p_i$ is the size of the
atom corresponding to the point mass at $x_i$. At any $t \geq0$ and
any clan indicator $z \in[0,1]$, the size of the clan at time $t$
corresponding to $z$ is just $\mu(t,E \times\{z\})$. The sum of all
the clan sizes is quite clearly $h_{\mathrm{eq}}$. If we normalize each
clan size by dividing it by $h_{\mathrm{eq}}$, then the normalized size
of the clan at time $t$ corresponding to $z$ is just $\nu(t,E \times\{
z\}) = \nu_c(t,\{z\})$. In other words, the normalized clan sizes at
time $t$ are nothing but the sizes of the atoms in $\nu_c(t)$. From now
on by \textit{clan size} we always mean the \textit{normalized clan size}.

The assertions of Proposition \ref{stationarity} are true for $\nu_c$
as well. Let $\Lambda_{\infty}$ be the infinite simplex given by
\[
\Lambda_{\infty}=\Biggl\{ (x_1,x_2,\ldots) \dvtx\sum_{i=1}^{\infty} x_i=1
\mbox{ and } 0<x_i<1, i=1,2,\ldots\Biggr\}.
\]
$\operatorname{GEM}(\theta)$ distribution is a distribution over the infinite simplex
$\Lambda_{\infty}$ that depends on a parameter $\theta\in[0,\infty)$.
This distribution is named after three population geneticists
McCloskey, Engen and Griffiths (see Johnson, Kotz and Balakrishnan \cite
{Bala} and Pitman and Yor \cite{Pitman}). It is defined as below.
\begin{definition}[{[$\operatorname{GEM}(\theta)$ distribution]}] \label{gem}
Let $\{W_n\dvtx n=1,2,\ldots\}$ be a sequence of i.i.d.
$\operatorname{Beta}(1,\theta)$
random variables [i.e., each $W_i$ has density $\theta(1-x)^{\theta
-1}$ for $0<x<1$]. Define $P_1=W_1$ and $P_n=(1-W_1)(1-W_2)\cdots
(1-W_{n-1})W_n$ for $n \geq1$. Then the sequence $\{P_n\dvtx
n=1,2,\ldots\}
$ is said to have the $\operatorname{GEM}(\theta)$ distribution.
\end{definition}

If we define
%
%
\begin{equation}
\label{deftheta}
\theta= \frac{k_{\mathrm{on}}}{k_{\mathrm{fb}}},
\end{equation}
then at stationarity the sizes of the atoms in $\nu_c(t)$ are
distributed according to the $\operatorname{GEM}(\theta)$ distribution. This is a
direct consequence of Theorem 4.6 in Chapter 10, Ethier and Kurtz \cite
{EK}. This result shows that at stationarity there are infinitely many
clans on the membrane and their sizes follow the $\operatorname{GEM}(\theta)$
distribution. If we arrange these sizes in descending order, then the
resulting random infinite vector has the Poisson--Dirichlet distribution
with the same parameter $\theta$ (see Chapter~2 in \cite{FengShui}).
The Poisson--Dirichlet distribution was introduced by Kingman \cite
{King75} in 1975 and many of its properties are well known. This
characterization of clan sizes at stationarity makes it possible to
compute the distribution and moments of the largest clan size, second
largest clan size, third largest clan size and so on (see Griffiths
\cite{Gr}). The joint distribution of the first few largest clans can
also be obtained (see Watterson \cite{Wat}). If we sample $n$ molecules
from the membrane at stationarity from $\nu(t)$, then the
distributional properties of the clans represented by this sample can
be studied via the Ewen's Sampling Formula (see \cite{Ewens}). All
these results indicate that the clan sizes at stationarity are far from
uniform and there are a few \textit{large} clans and many \textit{small}
clans. Most of the molecules are contained in these few large clans. In
fact, if we sample $n$ membrane molecules at stationarity, then they
would belong to roughly $\theta\log n$ distinct clans asymptotically
(see Theorem 2.11 in \cite{FengShui}).

The quantity $S_p$ [given by (\ref{spreadst})] measures the average
spatial spread of the clans and its value at stationarity is given by
Theorem \ref{theoremspread}. At stationarity there are only a few large
clans and if $S_p$ is small relative to the cell size, then the spatial
spread of these large clans is also small. Therefore, the distribution
of molecules at stationarity is highly asymmetrical at all times.

We now discuss the emergence of cell polarity. First we need to define
it mathematically.
\begin{definition}[($\varepsilon$-polarity)]
\label{defcellpolarity}
For any $ 0 < \varepsilon\ll1$, we say that the cell is $\varepsilon
$-polarized if at least $(1-\varepsilon)$ fraction of the membrane
population belongs to a single clan and also resides in a single
hemisphere on the membrane.
\end{definition}

The above definition is motivated by the biological literature (see
\cite{Cove,Cove2,AAWW}). Note that the molecules in a clan will
generally appear to cluster around the location of their most recent
common ancestor (see \cite{DH2}). Therefore, as in~\cite{AAWW}, if
diffusion is small, having one predominant clan on the membrane is
a~good indication that a single site of polarity has formed.

At stationarity, the probability that the cell is $\varepsilon$-polarized
at any time $t$ can be expressed as
%
%
\begin{eqnarray}
\label{defpepsilon}\qquad
p_{\varepsilon} & = & \Pi\bigl( \bigl\{ \beta\in\mathcal{P}(E \times
[0,1]) \mbox{: there exists a } z\in[0,1]\nonumber\\[-8pt]\\[-8pt]
&&\hspace*{15pt} \mbox{ and a hemisphere }  H \subset E \mbox{
such that } \beta(H \times\{z\}) \geq1-\varepsilon\bigr\} \bigr),
\nonumber
\end{eqnarray}
where $\Pi\in\mathcal{P}(\mathcal{P}(E \times[0,1]))$ is the
stationary distribution of the process $\nu$. We mentioned before that
at stationarity, the vector of clan sizes in descending order follows
the Poisson Dirichlet distribution with parameter $\theta$. Let $V_1$
be the size of the largest clan. For any $\varepsilon>0$, Theorem 2.5 in
\cite{FengShui} implies that
%
%
\begin{equation}
\label{defqepsilon}
q_{\varepsilon} := P \bigl( V_1 > \sqrt{1- \varepsilon} \bigr) > 0.
\end{equation}
Suppose that we are at stationarity. Let $r_{\varepsilon}$ denote the
probability that $(\sqrt{1- \varepsilon})$-fraction of the molecules in
the largest clan are situated in a single hemisphere on the membrane
given that the size of the largest clan is at least $\sqrt{1-
\varepsilon
}$. Observe that $S_p$ is like a weighted average of the spatial
spreads of the clans, where the weight of each clan is proportional to
its size. Therefore, if almost all the molecules are in the largest
clan, then $S_p$ is approximately the spatial spread of the largest
clan. Hence, if $S_p$ is small in comparison to the cell size, we can
reasonably expect $r_{\varepsilon}$ to be positive. Observe that
$p_{\varepsilon} \geq q_{\varepsilon} r_{\varepsilon}$ and so
$p_{\varepsilon}$ is
also positive for a small positive $\varepsilon$. Since the process $\nu$
is ergodic, Birkhoff's ergodic theorem (see Theorem 10.6 and Corollary
10.9 in \cite{kallenberg}) implies that the cell will definitely reach
the $\varepsilon$-polarized state and, in fact, spend $p_{\varepsilon}$
proportion of its time there. Thus, the cell gets $\varepsilon$-polarized
infinitely often and we have recurring cell polarity.

Before we proceed we need to define some new quantities. Let
%
%
\begin{eqnarray}
\label{defspbar}
\bar{S}_p &=& \frac{S_p}{R^2},
\\
\label{defvartheta}
\chi&=& \frac{D}{R^2}
\end{eqnarray}
and
%
%
\begin{equation}
\label{defgamma}
\gamma= k_{\mathrm{fb}} \biggl( \frac{1 - h_{\mathrm{eq}}}{h_{\mathrm
{eq}}} \biggr) = \biggl( \frac{k_{\mathrm{fb}} k_{\mathrm
{off}}}{k_{\mathrm{fb}}- k_{\mathrm{off}}}\biggr).
\end{equation}
We can interpret $\bar{S}_p$ as the average spatial spread of the clans
relative to the cell size, while $\chi$ can be seen as the speed of
diffusion relative to the cell size. Note that the ratio $(1-h_{\mathrm
{eq}})/h_{\mathrm{eq}}$ is the molecular mass available in the cytosol
for recruitment per membrane molecule. The parameter $\gamma$ is just
the feedback rate tempered by this availability ratio. It can be
interpreted as the effective feedback strength.
We can recast the result of Theorem \ref{theoremspread} as
%
%
\begin{equation}
\label{formbarsp}
\bar{S}_p = \frac{2 \chi}{ ( (1+\theta)\gamma+\chi)}.
\end{equation}

Recall that the biological parameters in our model are $D$, $R$,
$k_{\mathrm{on}}$, $k_{\mathrm{fb}}$ and $k_{\mathrm{off}}$. We now
examine their impact on cell polarity. Instead of working with the
original parameters, we will work with $\theta$, $\chi$ and $\gamma$.
From the above discussion it is clear that the formation of polarity
will be facilitated if the probability $q_{\varepsilon}$ [given by (\ref
{defqepsilon})] is high while the quantity $\bar{S}_p$ is low.
As noted earlier, the parameter $\theta$ controls the distribution of
molecules into the infinitely many clans present at stationarity. From
the properties of Poisson Dirichlet distributions we know that the
probability $q_{\varepsilon}$ decreases as~$\theta$ increases and
vice-versa (see \cite{DawsonFeng} and Chapter 2 in~\cite{FengShui}). In
fact, it can be shown that this probability is nearly $1$ if $\theta
\approx0$ (see \cite{FengShui2}). Hence, polarity will establish more
easily if $\theta$ is small. Recall that the process $\nu_c$ [given by
(\ref{clanprocess})] is the Fleming--Viot process corresponding to the
infinitely-many-neutral-alleles model. The sample paths of this process
take values over the space of atomic\vadjust{\goodbreak} measures over $[0,1]$. Using
Dirichlet forms, Schmuland \cite{Schmuland} has shown that with
probability $1$ there will exist times at which this process will hit
the state of having a single atom if and only if $\theta<1$. Therefore,
$\theta<1$ assures that there will exist times when there is only one
clan present. At these times the chances of observing polarity will be
nearly\vspace*{1pt} $1$ if $\bar{S}_p$ is sufficiently small. The formula (\ref
{formbarsp}) makes it clear that the quantity $\bar{S}_p$ gets smaller
as the relative diffusion speed ($\chi$) goes down or the effective
feedback strength ($\gamma$) goes up.

Recall that the likelihood of finding a cell in the $\varepsilon
$-polarized state at any time at stationarity is given by
$p_{\varepsilon
}$ [given by (\ref{defpepsilon})]. The observations made in the
preceding paragraph show that $p_{\varepsilon}$ increases with $\gamma$
but decreases with $\theta$ and $\chi$. Unfortunately, we do not have a
precise formula for $p_{\varepsilon}$ at the moment. Such a formula would
be really useful in determining the chances of observing polarity in a
cell with any given set of parameters. It will also give us a clear
idea of the time spent by the cell in the polarized state.

We would now like to compare our results to the results presented in
\cite{AAWW} for the original model. To avoid confusion, we will denote
the association, dissociation and recruitment rates in the original
model as $k^{o}_{\mathrm{on}}$, $k^{o}_{\mathrm{off}}$ and
$k^{o}_{\mathrm{fb}}$, respectively. Note that under our scaling
$k^{o}_{\mathrm{on}} = k_{\mathrm{on}}$ while $k^{o}_{\mathrm{off}} = N
k_{\mathrm{off}}$ and $k^{o}_{\mathrm{fb}} = N k_{\mathrm{fb}}$, where
$N$ is the total population size. The analysis in \cite{AAWW} assumes
that $k^{o}_{\mathrm{fb}}>k^{o}_{\mathrm{off}}$ and $k^{o}_{\mathrm
{on}}$ is much smaller in comparison to $k^{o}_{\mathrm{off}}$ or
$k^{o}_{\mathrm{fb}}$. Observe that spontaneous membrane association
tends to spatially homogenize the molecules on the membrane and so if
$k^{o}_{\mathrm{on}}$ is not small in comparison, we cannot hope to see
cell polarity. Under the above assumptions it is shown in \cite{AAWW}
that the fraction of molecules on the membrane approaches the
equilibrium value
\[
h_{\mathrm{eq}} = 1 -\frac{k^{o}_{\mathrm{off}}}{k^{o}_{\mathrm{fb}}}+
O\biggl( \frac{k^{o}_{\mathrm{on}}}{k^{o}_{\mathrm{fb}}} \biggr),
\]
at an exponential rate with half-time of $(h_{\mathrm{eq}}
k^{o}_{\mathrm{fb}})^{-1}$. In our scaling, $k^{o}_{\mathrm{on}} =
O(1)$, while $k^{o}_{\mathrm{fb}}$ and $k^{o}_{\mathrm{off}}$ are
$O(N)$. Therefore, it is not surprising that as $N \to\infty$ the
fraction of molecules on the membrane reaches the equilibrium value
\[
h_{\mathrm{eq}} = 1 - \frac{k_{\mathrm{off}}}{k_{\mathrm{fb}}}
\]
almost instantly [see part (A) of Theorem \ref{mainconvthm}]. Since
$k_{\mathrm{on}}$ is small, the bulk of the population at equilibrium
must come through membrane recruitment. It is mentioned in \cite{AAWW}
that if $k^{o}_{\mathrm{fb}} \leq k^{o}_{\mathrm{off}}$, the membrane
will be nearly empty at equilibrium and so clusters cannot form. For
the same reason Assumption~\ref{assmp1} is required in this paper.

As we have discussed above, the emergence of cell polarity crucially
depends on the likelihood of having just one large clan on the
membrane. It is shown in \cite{AAWW} that for a finite population size
$N$, the number of clans on the membrane will reduce to just $1$ at
certain times, giving rise to polarity (if $D$ is small), if the
spontaneous association events are rare ($k^{o}_{\mathrm{on}}$ is
small). However, the frequency at which these times arrive is
proportional to $1/N$ and, hence, there is no recurring polarity in the
infinite population limit unlike the rescaled model that we consider.

It is observed in \cite{AAWW} that the clustering behavior for the
original model is entirely determined by a simple relationship between
the ratio $\frac{k^{o}_{\mathrm{on}}}{k^{o}_{\mathrm{fb}}}$\vspace*{-1pt} and the
population size~$N$. Their analysis shows that if $\frac{k^{o}_{\mathrm
{on}}}{k^{o}_{\mathrm{fb}}} \ll N^{-2} $, then certainly one cluster
will form and if $\frac{k^{o}_{\mathrm{on}}}{k^{o}_{\mathrm{fb}}} \gg
(N^{-1}\log{N})^{1/2} $, then no clusters will form.\vspace*{-1pt} Using numerical
simulations, they observe that the transition occurs when $\frac
{k^{o}_{\mathrm{on}}}{k^{o}_{\mathrm{fb}}} \approx N^{-1}$. This
motivated us to scale $k^{o}_{\mathrm{fb}}$ by a factor of $N$ and
analyze the resulting model. We also had to scale $k^{o}_{\mathrm
{off}}$ by $N$ because otherwise the entire population will soon be on
the membrane (as $h_{\mathrm{eq}}$ will then be $1$), depleting the
cytosol and preventing further membrane recruitment. In such
a~scenario, the feedback mechanism will be unable to counter the surface
diffusion and, hence, there will not be any lasting cell polarity. Note
that having $k^{o}_{\mathrm{fb}} = N k_{\mathrm{fb}}$ is the same as
changing the feedback rate in Description~\ref{description} to
$k_{\mathrm{fb}} \times\mbox{(\textit{number of molecules in the
cytosol})}$. This is the same as saying that each membrane molecule
recruits each cytosol molecule at rate~$k_{\mathrm{fb}}$. Such
a~definition may be more natural for the feedback circuits found in
certain cells. Our results provide an explanation for the existence of
cell polarity in such cells if the population size is large.

There are many biologically appealing questions about the model that we
have been unable to answer in this paper. As we mentioned above, it
would be useful be have precise estimates for $p_{\varepsilon}$. It would
also be interesting to compute the time it takes to hit the $\varepsilon
$-polarized state and the time the cell stays polarized after that.
These results would give us a better idea about the the onset and
maintenance of polarity. The role of various model parameters will
emerge clearly as well.

The model we study does have the drawback of being simplistic, as all
the molecules in the cell are identical. Most cells that exhibit
polarity have molecules of many different types recruiting each other
at various type-dependent rates (see \cite{DN,AAWWref9,Takaku}). We
would like to know if a multi-type generalization of our polarity model
would also lead to tractable measure-valued dynamics in the infinite
population limit. We hope to answer this question elsewhere in the very
near future.

In this paper we have only looked at \textit{single-site} polarity. Many
cells exhibit \textit{anterior--posterior} polarity (see \cite
{Shapiro,Seydoux,Evans0}). In such cells there are usually two types of
molecules and they segregate themselves in such a way that one type of
molecule forms the \textit{front} and the other type of molecule forms
the \textit{rear}. Such an arrangement is vital for cell division and
locomotion. It has been suggested that this phenomenon is caused when
molecules not only recruit the molecules of their own type but also
locally inhibit the recruitment of the other type. It may be possible
to extend the model considered here to account for anterior--posterior
polarity as well.

The story of cell polarity is far from over and we hope that more work
will be done to mathematically understand this biologically vital
phenomenon and answer the challenging questions it poses.

\section{Proofs}\label{sec4}

In this section we prove the main results of our paper: Theorems \ref
{mainconvthm} and \ref{theoremspread}.

\subsection{\texorpdfstring{Proof of part \textup{(A)} of Theorem \protect\ref{mainconvthm}}%
{Proof of part (A) of Theorem 2.3}}
Recall that $n^N(t) = N \langle1,\mu^N(t)\rangle$ denotes
the number of molecules on the membrane at time $t$. Since $\mu^N$ has
generator~$\mathbb{A}^N$, we can write the generator $K^N$ for the
$\mathbb{N}_{0}$-valued process $n^N$ as the following. For $f \in
C( \mathbb{R})$ let
\begin{eqnarray*}
K^N f(n)& = &k_{\mathrm{on}}(N-n) \bigl(f(n+1)-f(n) \bigr)+
N k_{\mathrm{off}} n \bigl(f(n-1)-f(n)\bigr) \\
&&{} + k_{\mathrm{fb}} n(N-n) \bigl( f(n+1)-f(n) \bigr).
\end{eqnarray*}
We start with nothing on the membrane and, hence, $n^N(0)=0$. The form
of the generator $K^N$ allows us to write the equation for $n^N$ as
%
%
\begin{eqnarray}
\label{main}
n^N(t)& =& Y_1\biggl( k_{\mathrm{on}} \int_{0}^{t} \bigl(N - n^N(s)
\bigr)\,ds \biggr) - Y_2\biggl(N k_{\mathrm{off}} \int_{0}^{t} n^N(s)\,ds
\biggr)\nonumber\\[-8pt]\\[-8pt]
&&{} + Y_3\biggl( N k_{\mathrm{fb}} \int_{0}^{t} n^N(s)\biggl(1-\frac
{n^N(s)}{N} \biggr) \,ds \biggr). \nonumber
\end{eqnarray}

We would like to estimate the first time $n^N$ reaches a positive
fraction of the population size $N$. Pick an $\varepsilon>0$ such that
$k_{\mathrm{fb}}(1-\varepsilon) > k_{\mathrm{off}}$ and define
%
%
\begin{equation}
\label{rhoN}
\rho_{\varepsilon}^N= \inf{ \{t \geq0 \dvtx n^N(t) \geq N \varepsilon
\}}.
\end{equation}

\begin{lemma}
\label{lem1}
Let $\lambda=k_{\mathrm{fb}}(1-\varepsilon) - k_{\mathrm{off}}$. Then
\[
\lim_{N \to\infty} P\biggl( \rho_{\varepsilon}^N \leq\frac{2 \log
{N}}{\lambda N} \biggr)= 1.
\]
Moreover, $\rho_{\varepsilon}^N \to0$ a.s. as $N \to\infty$.
\end{lemma}
\begin{pf}
We first slow the time by a factor of $N$. Let $\tilde
{n}^N(t)=n^N(t/N)$. Since $n^N$ satisfies equation (\ref{main}), $\tilde
{n}^N $ satisfies
%
%
\begin{eqnarray}
\label{smain}
\tilde{n}^N(t) &=& Y_1\biggl(k_{\mathrm{on}} \int_{0}^{t} \biggl(1 - \frac
{\tilde{n}^N(s)}{N}\biggr)\,ds \biggr) - Y_2\biggl( k_{\mathrm{off}} \int
_{0}^{t} \tilde{n}^N(s)\,ds \biggr) \nonumber\\[-8pt]\\[-8pt]
&&{} +Y_3\biggl( k_{\mathrm{fb}} \int
_{0}^{t} \tilde{n}^N(s)\biggl(1-\frac{\tilde{n}^N(s)}{N} \biggr) \,ds
\biggr).\nonumber
\end{eqnarray}
Define
%
%
\begin{equation}
\label{rhoNt}
\tilde{\rho}_{\varepsilon}^N= \inf{ \{t \geq0 \dvtx\tilde{n}^N(t) \geq N
\varepsilon\}}=N \rho_{\varepsilon}^N.
\end{equation}
To prove the first claim of the lemma, we only need to show that
%
%
\begin{equation}
\label{suffc}
\lim_{N \to\infty} P\biggl( \tilde{\rho}_{\varepsilon}^N \leq\frac{2 \log
{N}}{\lambda} \biggr)= 1.
\end{equation}

For $0 \leq t < \tilde{\rho}_{\varepsilon}^N $, $\frac{\tilde{n}^N(t)}{N}
\leq\varepsilon$. Define another process $Z$ by the equation
%
%
\begin{eqnarray}
\label{sdez}
Z(t) &=& Y_1\bigl(k_{\mathrm{on}}(1-\varepsilon)t \bigr)- Y_2 \biggl(
k_{\mathrm{off}} \int_{0}^{t} Z(s)\,ds \biggr) \nonumber\\[-8pt]\\[-8pt]
&&{}+Y_3\biggl( k_{\mathrm
{fb}}(1-\varepsilon) \int_{0}^{t} Z(s)\,ds \biggr).\nonumber
\end{eqnarray}
Note that $Z$ is independent of $N$ and $\varepsilon$ is chosen so that
$k_{\mathrm{fb}}(1-\varepsilon) > k_{\mathrm{off}}$. The form of the
equation for $Z$ shows that $Z$ is a supercritical branching process
with immigration. For $0 \leq t < \tilde{\rho}_{\varepsilon}^N $ we
clearly have $Z(t) \leq\tilde{n}^N(t) < \varepsilon N$.
Define
%
%
\begin{equation}
\label{rhob}
\bar{\rho}_{\varepsilon}^N= \inf{ \{t \geq0 \dvtx Z(t) \geq N
\varepsilon\}}.
\end{equation}
It is easy to see that $\tilde{\rho}_{\varepsilon}^N \leq\bar{\rho
}_{\varepsilon}^N$. We will find a probabilistic upper bound on~$\bar
{\rho
}_{\varepsilon}^N$ which will show (\ref{suffc}) and hence prove the first
claim of the lemma.

A supercritical branching process with immigration can be written as
a~superposition of independent supercritical branching processes starting
with an initial population of $1$ at various times. This fact along
with Theorems 1 and 2 in Chapter~3, Section 7, in Athreya and Ney \cite
{AN} show that there exists a random variable $W$ such that $W>0$ a.s. and
\[
\lim_{t \to\infty} e^{-\lambda t} Z(t) = W \qquad\mbox{a.s.}
\]
Therefore,
\[
\lim_{N \to\infty} e^{-\lambda\bar{\rho}_{\varepsilon}^N} Z(\bar{\rho
}_{\varepsilon}^N) = W \qquad\mbox{a.s.},
\]
which implies that
\[
\lim_{N \to\infty} \log(e^{-\lambda\bar{\rho}_{\varepsilon}^N}
Z(\bar{\rho}_{\varepsilon}^N) )=\lim_{N \to\infty} \bigl(-\lambda
\bar{\rho}_{\varepsilon}^N+\log Z(\bar{\rho}_{\varepsilon}^N) \bigr)= \log W
\qquad\mbox{a.s.}
\]
Observe that $ N \varepsilon\leq Z(\bar{\rho}_{\varepsilon}^N)\leq
N\varepsilon
+1$. From above we get
\[
\lim_{N \to\infty} \frac{\bar{\rho}_{\varepsilon}^N}{\log N} = \frac{1}{
\lambda} \qquad\mbox{a.s.}
\]
Since $N \rho_{\varepsilon}^N=\tilde{\rho}_{\varepsilon}^N \leq\bar{\rho
}_{\varepsilon}^N$ a.s., the above limit implies (\ref{suffc}) and also
shows that $\rho_{\varepsilon}^N \to0$ a.s. as $N \to\infty$. This
completes the proof of the lemma.
\end{pf}

Recall the definition of equilibrium fraction $h_{\mathrm{eq}}$ from
the statement of Theorem~\ref{mainconvthm}. Fix\vadjust{\goodbreak} $\varepsilon$ to be
$\frac
{h_{\mathrm{eq}}}{2} = \frac{1}{2} ( 1-\frac{k_{\mathrm
{off}}}{k_{\mathrm{fb}}} )$ and let $\rho^N$ be $\rho_{\varepsilon
}^N$ for this particular choice of $\varepsilon$.
By Lemma \ref{lem1} we obtain
%
%
\begin{equation}
\label{plimrho}
\lim_{N \to\infty} P\biggl( \rho^N \leq\frac{4 \log{N}}{ (k_{\mathrm
{fb}}-k_{\mathrm{off}})N} \biggr)= 1.
\end{equation}

Recall that the process $h^N$ is given by (\ref{defhn}). Using
(\ref{main}), we can write an equation for $h^N$ as
%
%
\begin{eqnarray}
\label{mainh}\quad
h^N(t) &=& \frac{1}{N} Y_1\biggl( N k_{\mathrm{on}} \int_{0}^{t} \bigl(1 -
h^N(s)\bigr)\,ds \biggr) - \frac{1}{N} Y_2\biggl(N^2 k_{\mathrm{off}} \int
_{0}^{t} h^N(s)\,ds \biggr)\nonumber\\[-8pt]\\[-8pt]
&&{} + \frac{1}{N} Y_3\biggl( N^2 k_{\mathrm{fb}} \int_{0}^{t} h^N(s)
\bigl(1-h^N(s) \bigr) \,ds \biggr). \nonumber
\end{eqnarray}

Let $\bar{h}^N$ be the process given by
%
%
\begin{equation}
\label{barhn}
\bar{h}^N(t) = h^N ( t + \rho^N ).
\end{equation}
Note that
%
%
\begin{equation}
\label{initialvalueofh}
\bar{h}^N(0) = h^N(\rho^N)=\frac{ \lceil N {h_{\mathrm
{eq}}}/{2} \rceil}{N}.
\end{equation}
For $i=1,2,3$ let
\[
\bar{Y}_i (t) = Y_i (t + \delta_i^N ) - Y_i (\delta_i^N
),
\]
where
\begin{eqnarray*}
\delta^N_1 &=& N k_{\mathrm{on}} \int_{0}^{\rho^N} \bigl(1 - h^N(s)\bigr)\,ds,
\\
\delta^N_2 &=& N^2 k_{\mathrm{off}} \int_{0}^{\rho^N} h^N(s)\,ds
\end{eqnarray*}
and
\[
\delta^N_3 = N^2 k_{\mathrm{fb}} \int_{0}^{\rho^N} h^N(s)\bigl(1-h^N(s)
\bigr) \,ds.
\]
From the strong Markov property of the Poisson process we can conclude
that $\bar{Y}_1$, $\bar{Y}_2$ and $\bar{Y}_3$ are independent\vspace*{1pt} unit
Poisson processes.
We can write the equation for process $\bar{h}^N$ as
%
%
\begin{eqnarray}
\label{mainhb}
\bar{h}^N(t) & = & \bar{h}^N(0) + h^N(t+\rho^N) - h^N(\rho^N) \nonumber\\
& = &\bar{h}^N(0) + \frac{1}{N} \bar{Y}_1\biggl( k_{\mathrm{on}}N \int
_{0}^{t} \bigl(1 - \bar{h}^N(s)\bigr)\,ds \biggr)\nonumber\\[-8pt]\\[-8pt]
&&{} - \frac{1}{N} \bar
{Y}_2\biggl(N^2 k_{\mathrm{off}} \int_{0}^{t} \bar{h}^N(s)\,ds
\biggr)\nonumber\\
&&{} + \frac{1}{N} \bar{Y}_3\biggl( N^2 k_{\mathrm{fb}} \int_{0}^{t} \bar
{h}^N(s)\bigl(1-\bar{h}^N(s) \bigr) \,ds \biggr). \nonumber
\end{eqnarray}
Let $\bar{Y}^{c}_i$ be the centered version of $\bar{Y}_i$ [i.e.,
$\bar{Y}^{c}_i(u) = \bar{Y}_i(u)-u$ for $u\geq0$]. Define
%
%
\begin{eqnarray}
\label{mart}\qquad
M_N(t) &=& \frac{1}{N} \bar{Y}^{c}_1\biggl( k_{\mathrm{on}}N \int_{0}^{t}
\bigl(1 - \bar{h}^N(s)\bigr)\,ds \biggr) - \frac{1}{N} \bar{Y}^{c}_2
\biggl(N^2 k_{\mathrm{off}} \int_{0}^{t} \bar{h}^N(s)\,ds
\biggr)\nonumber\\[-8pt]\\[-8pt]
&&{} + \frac{1}{N} \bar{Y}^{c}_3 \biggl( N^2 k_{\mathrm{fb}} \int_{0}^{t}
\bar{h}^N(s)\bigl(1-\bar{h}^N(s) \bigr) \,ds \biggr), \nonumber
\end{eqnarray}
which is a martingale with quadratic variation given by
%
%
\begin{eqnarray}
\label{qv}
[M_N]_t &=& \frac{1}{N^2} \bar{Y}_1 \biggl( k_{\mathrm{on}}N \int_{0}^{t}
\bigl(1 - \bar{h}^N(s)\bigr)\,ds \biggr) + \frac{1}{N^2} \bar{Y}_2
\biggl(N^2 k_{\mathrm{off}} \int_{0}^{t} \bar{h}^N(s)\,ds
\biggr)\hspace*{-28pt}\nonumber\\[-8pt]\\[-8pt]
&&{} + \frac{1}{N^2} \bar{Y}_3\biggl( N^2 k_{\mathrm{fb}} \int_{0}^{t} \bar
{h}^N(s)\bigl(1-\bar{h}^N(s) \bigr) \,ds \biggr). \nonumber
\end{eqnarray}
Since $0 \leq\bar{h}^N\leq1$, we have
%
%
\begin{equation}
\label{eqv}
E([M_N]_t)\leq k_{\mathrm{on}} \frac{t}{N} +k_{\mathrm{off}}t +
k_{\mathrm{fb}}t.
\end{equation}
By centering the Poissons in equation (\ref{mainhb}), we can write
%
%
\begin{eqnarray}
\label{mainhb1}
\bar{h}^N(t)&=&\bar{h}^N(0)+ k_{\mathrm{on}} \int_{0}^{t} \bigl(1 - \bar
{h}^N(s)\bigr)\,ds - N k_{\mathrm{off}} \int_{0}^{t} \bar{h}^N(s)\,ds
\nonumber\\[-8pt]\\[-8pt]
&&{} + N k_{\mathrm{fb}} \int_{0}^{t} \bar{h}^N(s)\bigl(1-\bar{h}^N(s)
\bigr) \,ds +M_N(t). \nonumber
\end{eqnarray}

Let $F(h)=k_{\mathrm{fb}}h(1-h)-k_{\mathrm{off}}h$ and
\[
Z_N(t)=\int_{0}^{t}k_{\mathrm{on}}\bigl(1 - \bar{h}^N(s)\bigr)\,ds + M_N(t).
\]
From (\ref{eqv}) and Corollary 2.3.3 in \cite{Metivier} we can conclude
that $\{Z_N\}$ is a~sequence of semimartingales that is relatively
compact in $D_{\mathbb{R}}[0,\infty)$. The jumps in $Z_N$ are of size
$1/N$ and, hence, if $Z$ is a limit point of this sequence, then~$Z$
must be a continuous process a.s.
equation (\ref{mainhb1}) can be written as
%
%
\begin{equation}
\label{hb2}
\bar{h}^N(t)=\bar{h}^N(0)+ Z_N(t)+ N \int_{0}^{t} F(\bar{h}^N(s))\,ds.
\end{equation}
Define another process $\alpha^N$ by
\[
\alpha^N(t) = \bar{h}^N(t) - h_{\mathrm{eq}}.
\]
Observe that
\begin{eqnarray*}
F(\bar{h}^N(t)) & = & k_{\mathrm{fb}}\bigl(\alpha^N(t) +h_{\mathrm
{eq}}\bigr)\bigl(1-h_{\mathrm{eq}}-\alpha^N(t)\bigr)-k_{\mathrm{off}}\bigl(\alpha
^N(t)+h_{\mathrm{eq}}\bigr) \\
&=&\alpha^N(t) \bigl(k_{\mathrm{fb}}(1-h_{\mathrm{eq}})-k_{\mathrm
{off}}\bigr)+k_{\mathrm{fb}}h_{\mathrm{eq}}(1-h_{\mathrm{eq}}) \\
&&{} -k_{\mathrm{off}}h_{\mathrm{eq}} -k_{\mathrm{fb}}\alpha^N(t)\bigl(\alpha
^N(t) +h_{\mathrm{eq}}\bigr) \\
&=& -k_{\mathrm{fb}}\alpha^N(t) \bar{h}^N(t)\qquad \biggl(
\mbox{using }
h_{\mathrm{eq}}=1-\frac{k_{\mathrm{off}}}{k_{\mathrm{fb}}}\biggr).
\end{eqnarray*}
From (\ref{hb2}) we get
\[
\alpha^N(t)= \alpha^N(0)- N k_{\mathrm{fb}} \int_{0}^{t}\bar
{h}^N(s)\alpha^N(s)\,ds +Z_N(t),
\]
which can be written in differential form as
\[
d \alpha^N(t) + N k_{\mathrm{fb}} \bar{h}^N(t)\alpha^N(t)\,dt = d Z_N(t).
\]
Let $\beta_N(t) = N k_{\mathrm{fb}} \int_{0}^{t}\bar{h}^N(s)\,ds$. Then
\[
d \bigl( \alpha^N(t) e^{\beta_N(t)}\bigr) = e^{\beta_N(t)} \,d Z_N(t).
\]
Integrating from $0$ to $t$, we get
\[
\alpha^N(t) e^{\beta_N(t)} - \alpha^N(0) = \int_{0}^{t}e^{\beta_N(s)} \,d Z_N(s).
\]
Therefore,
%
%
\begin{equation}
\label{alpha1}
\alpha^N(t) - \alpha^N(0) e^{ -\beta_N(t)} = e^{ -\beta_N(t)} \int
_{0}^{t} e^{\beta_N(s)} \,dZ_N(s).
\end{equation}

Let
\[
\bar{\sigma}^N=\inf\biggl\{ t \geq0\dvtx\bar{h}^N(t) \leq\frac
{h_{\mathrm{eq}}}{4} \biggr\}.
\]
For $0\leq t \leq\bar{\sigma}^N$
\[
\bar{h}^N(t) \geq\frac{h_{\mathrm{eq}}}{4} - \frac{1}{N}
\]
and, hence, for any small positive number $\varepsilon$,
\[
\inf_{0\leq t \leq(T \wedge\bar{\sigma}^N -\varepsilon) } | \beta
_{N}(t+\varepsilon) - \beta_{N}(t) | \Rightarrow\infty
\]
as $N \to\infty$. Fix any $T>0$. By Lemma 5.2 in \cite{kat},
\[
\sup_{0 \leq t \leq T \wedge\bar{\sigma}^N } e^{ -\beta_N(t)} \biggl|
\int_{0}^{t} e^{\beta_N(s)} \,d Z_N(s)\biggr| \Rightarrow0.
\]
Hence, (\ref{alpha1}) gives
%
%
\begin{equation}
\label{alpha2}
\sup_{0 \leq t \leq T \wedge\bar{\sigma}^N } \bigl| \alpha^N(t) -
\alpha^N(0) e^{ -\beta_N(t)} \bigr| \Rightarrow0.
\end{equation}
If $\bar{\sigma}^N < T$, then from the definitions of $\bar{\sigma}^N$,
$\alpha^N$ and (\ref{initialvalueofh}) we get
\begin{eqnarray*}
\sup_{0 \leq t \leq T \wedge\bar{\sigma}^N } \bigl| \alpha^N(t) -
\alpha^N(0) e^{ -\beta_N(t)} \bigr| & \geq &\bigl| \alpha^N(\bar{\sigma
}^N) - \alpha^N(0) e^{ -\beta_N(\bar{\sigma}^N)} \bigr| \\[-2pt]
& \geq &\biggl| - \frac{3 h_{\mathrm{eq}}}{4} + \frac{h_{\mathrm
{eq}}}{2} e^{ -\beta_N(\bar{\sigma}^N)} \biggr| -\frac{2}{N} \\[-2pt]
& \geq &\frac{h_{\mathrm{eq}}}{4} - \frac{2}{N}.
\end{eqnarray*}
This calculation shows that
\[
P ( \bar{\sigma}^N < T ) \leq P \biggl( \sup_{0 \leq t \leq T
\wedge\bar{\sigma}^N } \bigl| \alpha^N(t) - \alpha^N(0) e^{ -\beta
_N(t)} \bigr| \geq\frac{h_{\mathrm{eq}}}{4} - \frac{2}{N} \biggr).
\]
Therefore, from (\ref{alpha2}), $P ( \bar{\sigma}^N < T )
\to0$ for any $T>0$ and, hence, $\bar{\sigma}^N \to\infty$ in probability.

Let the stopping time $\tau_N$ be given by
%
%
\begin{equation}
\label{deftaun}
\tau_N = \rho^N + \frac{\log N}{N}.
\end{equation}
From Lemma \ref{lem1} $\rho^N \to0 $ a.s. and, hence, $\tau_N \to0 $
a.s. Let the process $\hat{h}^N$ be defined by
%
%
\begin{equation}
\label{hath}
\hat{h}^N (t) = h^N(t+\tau_N) = \bar{h}^N\biggl(t+\frac{\log
N}{N}\biggr),\qquad t \geq0.
\end{equation}
Note that
\begin{eqnarray*}
&& \sup_{0 \leq t \leq( (T \wedge\bar{\sigma}^N) - ({\log
N})/{N} ) } | \hat{h}^N (t)- h_{\mathrm{eq}} | \\[-2pt]
&&\qquad = \sup
_{0 \leq t \leq( (T \wedge\bar{\sigma}^N) - ({\log N})/{N}
)} \biggl| \alpha^N\biggl(t+ \frac{\log N}{N}\biggr) \biggr| \\[-2pt]
&&\qquad \leq\sup_{0 \leq t \leq((T \wedge\bar{\sigma}^N) -
({\log N})/{N} ) } \biggl| \alpha^N\biggl(t+ \frac{\log N}{N}\biggr) -
\alpha^N(0) e^{ -\beta_N(t+({\log N})/{N})} \biggr|\\[-2pt]
&&\qquad\quad{} + \sup_{0 \leq t \leq((T \wedge\bar{\sigma}^N) - ({\log
N})/{N} )} \bigl| \alpha^N(0) e^{ -\beta_N(t+({\log
N})/{N})} \bigr|.
\end{eqnarray*}
The first term converges to $0$ in probability due to (\ref{alpha2}).
Observe that for $0 \leq t \leq((T \wedge\bar{\sigma}^N) - \frac
{\log N}{N} )$,
\[
\beta_N\biggl(t+\frac{\log N}{N}\biggr) \geq N k_{\mathrm{fb}} \int
_{0}^{t+({\log N})/{N}}\bar{h}^N(s)\,ds = k_{\mathrm{fb}} \biggl( \frac
{h_{\mathrm{eq}}}{4}-\frac{1}{N} \biggr)\log N .
\]
Thus, the second term above converges to $0$ as $N \to\infty$ a.s.
Since $\log N/N \to0$ and $\bar{\sigma}^N \to\infty$, we get
\[
\sup_{0 \leq t \leq T } | \hat{h}^N (t)- h_{\mathrm{eq}} |
\Rightarrow0 .
\]
This proves part (A) of Theorem \ref{mainconvthm}.\vadjust{\goodbreak}

\subsection{\texorpdfstring{Proof of part \textup{(B)} of Theorem \protect\ref{mainconvthm}}%
{Proof of part (B) of Theorem 2.3}}
To prove part (B) of the theorem, we will use the particle construction
introduced by Donnelly and Kurtz in \cite{DK99}. In this construction
the molecules are arranged in levels which are indexed by positive
integers. The arrangement is such that for any positive integer~$n$,
the process determined by the first $n$ levels is embedded in the
process determined by the first $(n+1)$ levels. This allows us to pass
to the projective limit. Another advantage of this construction is that
it makes the ancestral relationships between molecules explicit. For
any set of molecules we can trace back their genealogical tree to
obtain results about the measure-valued process.

We first motivate the particle construction. Suppose the total
population size is $N$ and at any time $t$ there are $n^N(t)$ molecules
on the membrane. The process $n^N$ follows equation (\ref{main}) and
suppose its evolution is known. Each molecule has a type in $E \times
[0,1]$ as before. We can represent the population on the membrane at
time $t$ by a vector $(Y^N_1(t),Y^N_2(t),\ldots,Y^N_{n^N(t)}
)$. Since the labeling of the molecules is arbitrary, it contains
exactly the same information as the measure $\tilde{Z}(t)=\sum
_{i=1}^{n} \delta_{Y_i(t)}$. We can choose any labeling we find
convenient. So we look into the future and order the individuals
according to the time of survival of their lines of descent. In this
new ordering we arrange the molecules into \textit{levels}, which are
taken to be positive integers. At any time~$t$, if there are $n^N(t)$
molecules, we will represent the population as the vector $
(X^N_1(t),X^N_2(t),\ldots,X^N_{n^N(t)})$. We will\vspace*{-1pt} refer
to~$X^N_i$ as the $i$th level process, where $X^N_i(t) \in E \times[0,1]$
is the molecule type at level $i$ at time $t$. Molecules are allowed to
change levels with time.
If a~death happens at time~$t$, then $n^N(t)=n^N(t-)-1$ and we just
remove the molecule at the highest index $n^N(t)$. If an immigration
happens at time~$t$, then $n^N(t)=n^N(t-)+1$ and we uniformly select a
level from the first $n^N(t-)+1$ levels and insert the immigrant
molecule there. If a birth event happens at time $t$, then
$n^N(t)=n^N(t-)+1$ and we do the following. We first uniformly select
two levels $i$ and $j$ from the first $n^N(t-)+1$ levels. Suppose $i$
is the smaller of the two levels. Then we shall refer to the molecule
$X^N_i(t-)$ as the parent and insert a copy of it at level $j$. The
molecules $\{ X^N_k(t-) \dvtx k= j ,j+1,\ldots\}$ are shifted up by one
level. So at time $t$, the offspring molecule $X^N_j(t)$ is a copy of
$X^N_i(t-)$, while $X^N_k(t) = X^N_k(t-)$ for $k<j$ and $X^N_k(t) =
X^N_{k-1}(t-)$ for $k>j$. In between all these events molecules are
doing speed $D$ Brownian motion on $E$ and changing their location.

What we have described above is a Markov process $X^N$ with state space
\[
S^N=\bigcup_{n=0}^{N} ( E \times[0,1] )^n.
\]
We adopt the convention that $( E \times[0,1] )^{0}=\{
\triangle\}$. For $x\in S^N$, if $x \in( E \times[0,1]
)^n$, then let $|x|=n$ for any $n=0,1,\ldots,N$. If at time $t$,
$X^N(t) = x\in S^N$ and $|x|=n$, then it means that there are $n$
molecules on the membrane with the type vector $x=(x_1,x_2,\ldots,x_n)
\in( E \times[0,1] )^n$.

If $|x|=n \geq m$ and $x=(x_1,x_2,\ldots,x_n)$, then let
$x^{|m}=(x_1,x_2,\ldots,x_m)$. Any $f \in B((E \times
[0,1])^m)$ can be regarded as a function over $S^N$ by defining
it as $f(x)=0$ if $|x|<m$ and $f(x)=f(x^{|m})$ if $|x|\geq m$. We now
specify the generator $A^N$ of the Markov process $X^N$ by its action
on functions in its domain $\mathcal{D}(A^N)=\mathcal{C}$ (see
Definition \ref{defc2}) as
%
%
\begin{eqnarray}
\label{pAN}\quad
A^N f(x) &= &\frac{D}{2} \sum_{i=1}^n \Delta_{i} f(x) + n N
k_{\mathrm{off}} \bigl(f(d_n(x)) - f(x)\bigr) \nonumber\\
&&{} + k_{\mathrm{on}} \biggl( \frac{N-n}{n+1} \biggr) \sum
_{i=1}^{n+1} \int_{E} \int_{0}^{1} \bigl( f( \theta_{i} (x ,(y,r) ) )-f(x)\bigr)
\vartheta(dy)\,dr \\
&&{} + 2 k_{\mathrm{fb}} \biggl( \frac{N-n}{n+1} \biggr) \sum
_{1 \leq i < j \leq(n+1)} \bigl(f( \theta_{ij} (x))-f(x)\bigr), \nonumber
\end{eqnarray}
where $n=|x|$ and if $x=(x_1,x_2,\ldots,x_n)$, then
$d_i(x)=(x_1,x_2,\ldots,x_{i-1},x_{i+1},\allowbreak\ldots,x_{n})$
(remove the $i$th coordinate), $\theta_{ij}(x)= (x_1,\ldots
,x_{j-1},x_{i},x_{j},\ldots,x_{n})$ (insert a copy of $x_i$ at the
$j$th place) and $\theta_{i}(x ,(y,r))=(x_1,\ldots
,x_{i-1},(y,r),\allowbreak x_{i},\ldots,x_{n})$ [insert $(y,r)$ at the $i$th place].

Viewing the operator $A^N$ as a bounded perturbation of the diffusion
operator [given by the first term on the right of (\ref{pAN})], we can
argue that the martingale problem for $A^N$ is well posed in the same
way we argued for~$\mathbb{A}^N$ in Section \ref{sec2}.
We now relate any solution of the
martingale problem for~$A^N$ to a solution of the martingale problem
for $\mathbb{A}^N$ [see (\ref{genpm1})] by using the Markov mapping
theorem (see Theorem 2.7 in Kurtz \cite{K98}). Let
\[
S^N_0=\mathcal{M}^N_a(E \times[0,1])=\Biggl\{ \frac{1}{N} \sum
_{i=1}^{n} \delta_{x_i} \dvtx0 \leq n \leq N \mbox{ and } x_1,\ldots
,x_n \in E \times[0,1] \Biggr\}
\]
and
\[
S^N=\bigcup_{n=0}^{N} ( E \times[0,1] )^n
\]
as before. Define $\gamma\dvtx S^N \to S^N_0$ by
\[
\gamma(x)=\frac{1}{N} \sum_{i=1}^{n} \delta_{x_i} \qquad\mbox{if }
x=(x_1,x_2,\ldots,x_n).
\]
Define the transition function $\alpha\dvtx S^N_0 \to\mathcal{P}(
S^N )$ by
\[
\alpha\Biggl( \frac{1}{N} \sum_{i=1}^n \delta_{x_i},dz \Biggr) = \frac
{1}{n!} \sum_{\sigma\in\Sigma_n} \delta_{(x_{\sigma(1)},x_{\sigma
(2)},\ldots,x_{\sigma(n)}) } \,dz,
\]
where $\Sigma_n$ is the set of all permutations on
$\{1,2,\ldots,n\}$.\vadjust{\goodbreak}
\begin{lemma}
\label{usingmm}
Let\vspace*{-1pt} $\pi^N_0 \in\mathcal{P}(S^N_0)$ and define $\pi^N=\int_{S^N_0}
\alpha(y,\cdot) \pi^N_0(dy)$. If $\nu^N$ is the solution of the
martingale problem for $(\mathbb{A}^N,\pi^N_0)$ and $X^N$ is the
solution of the martingale problem for $(A^N,\pi^N)$, then $\gamma
(X^N)$ and $\nu^N$ have the same distribution in $D_{\mathcal{M}^N_a(E
\times[0,1])}[0,\infty)$. Furthermore, for any $t \geq0$ the
distribution of $X^N(t) = (X^N_1(t),X^N_2(t),\ldots)$ is
exchangeable.
\end{lemma}
\begin{remark}
The length of the vector $X^N(t)$ is $n^N(t)$, which is a~\mbox{random}
variable. When\vspace*{1pt} we say that the distribution of $X^N(t) =
(X^N_1(t),X^N_2(t),\ldots)$ is exchangeable we mean that given
$n^N(t)=n$, the distribution of $(X^N_1(t),\allowbreak X^N_2(t),\ldots
,X^N_n(t) )$ is exchangeable.
\end{remark}
\begin{pf*}{Proof of Lemma \ref{usingmm}}
The definition of $\alpha$ ensures that $\alpha(\mu,\gamma^{-1}(\mu
))=1$ for all $\mu\in S^N_0$. If $f \in\mathcal{C}\cap C ((
E \times[0,1] )^{m} )$ and $\mu=\frac{1}{N} \sum_{i=1}^n
\delta_{x_i}$, then let
%
%
\begin{equation}\qquad
F(\mu)=\int_{S^N} f(z)\alpha(\mu,dz) = \frac{1}{n!} \sum_{\sigma\in
\Sigma_n} f\bigl(x_{\sigma(1)},\ldots,x_{\sigma(n)}\bigr)= \bigl\langle f,\mu
^{(m)} \bigr\rangle.
\end{equation}
Hence, $F \in\bar{\mathcal{C}}=\mathcal{D}(\mathbb{A})$. Now we show
that for such a function $F$
%
%
\begin{equation}
\label{mmrelation}
\mathbb{A}^N F(\cdot)=\int_{S^N}A^N f(z)\alpha(\cdot,dz).
\end{equation}
On writing down the expressions for $\mathbb{A}^N$ and $A^N$ using (\ref
{genpm1}) and (\ref{pAN}), we observe that there are four terms on each
side of (\ref{mmrelation}). We will show that the equality holds term
by term. It is easy to see that the first term corresponding to the
Brownian diffusion of membrane molecules is equal on both sides. We
check the equality for the next three terms below.

For $x=(x_1,\ldots,x_n)$ let $\sigma(x) =(x_{\sigma(1)},x_{\sigma
(2)},\ldots,x_{\sigma(n)})$ for any $\sigma\in\Sigma_n$. Let
$\mu=\frac{1}{N} \sum_{i=1}^n \delta_{x_i}$. Then
%
%
\begin{eqnarray}
\label{immi}
F\biggl(\mu+\frac{1}{N} \delta_{(y,r)} \biggr) &=& \frac{1}{(n+1)!} \sum
_{\sigma\in\Sigma_{n+1}} f( \sigma( \theta_{n+1} (x,(y,r))
) ) \nonumber\\[-9.5pt]\\[-9.5pt]
&=& \frac{1}{n+1} \sum_{i=1}^{n+1} \frac{1}{n!} \sum_{\sigma\in\Sigma
_n} f(\theta_{i} (\sigma(x),(y,r))).\nonumber
\end{eqnarray}
Similarly,
%
%
\begin{eqnarray}
\label{birth}\quad
N \int_{E \times[0,1]} F \biggl( \mu+\frac{1}{N} \delta_{x} \biggr)\mu
(dx)&=& \sum_{i=1}^{n} \frac{1}{(n+1)!} \sum_{\sigma\in\Sigma_{(n+1)}}
f( \sigma( \theta_{i n^{+}}(x) ) ) \nonumber\\[-2pt]
&=& \frac{1}{(n+1)!} \sum_{i=1}^{n} \sum_{j \neq i} \sum_{\sigma\in
\Sigma_n} f(\theta_{ij} (\sigma(x))) \\[-2pt]
&=& \frac{2}{n+1} \sum_{1 \leq i < j \leq(n+1)} \frac{1}{n!} \sum
_{\sigma\in\Sigma_n} f(\theta_{ij} (\sigma(x))),\nonumber
\end{eqnarray}
where $n^{+} = (n+1)$ in the first equation above.
Finally,
%
%
\begin{eqnarray}
\label{death}
N \int_{E \times[0,1]} F \biggl( \mu-\frac{1}{N} \delta_{x} \biggr)\mu
(dx)& =& \sum_{i=1}^{n} \frac{1}{(n-1)!} \sum_{\sigma\in\Sigma_{n-1}}
f( \sigma( d_i(x) ) ) \nonumber\\
&=& \frac{1}{(n-1)!} \sum_{\sigma\in\Sigma_n} f( \sigma(
d_n(x) ) ) \\
&=&n \frac{1}{n!} \sum_{\sigma\in\Sigma_n} f( \sigma(
d_n(x) ) ).\nonumber
\end{eqnarray}
Equations (\ref{immi}), (\ref{birth}) and (\ref{death}) show that the
relation (\ref{mmrelation}) holds and so the Markov mapping theorem is
applicable. Therefore, we can conclude that $\gamma(X^N)$ and $\nu^N$
have the same distribution in $D_{\mathcal{M}^N_a(E \times
[0,1])}[0,\infty)$.
From Corollary 3.5 in Kurtz \cite{K98} we obtain that if $n^N(t)=n$, then
\[
E(f(X^N_1(t),X^N_2(t),\ldots,X^N_{n}(t))|\mathcal{F}_{t}
)=\int_{S^N} f(z) \alpha(\gamma(X^N(t)),dz),
\]
where $\{\mathcal{F}_{t}\}$ is the filtration generated by the process
$\gamma(X^N(\cdot))$. Since $\alpha$ is symmetric, the distribution of
$(X^N_1(t),X^N_2(t),\ldots)$ is exchangeable.
\end{pf*}

Recall from Section \ref{sec2} that $\bar{\pi}_0 \in\mathcal{P}(\mathcal
{M}^N_{a}(E \times[0,1]))$ is the distribution that puts all the
mass at the $0$ measure and $\mu^N$ [given by (\ref{defmun})] is the
unique solution to the martingale problem\vspace*{1pt} corresponding to $
(\mathbb{A}^N,\bar{\pi}_0)$. For $\tau_N$ given by (\ref
{deftaun}), define the process $\hat{\mu}^N$ by
%
%
\begin{equation}
\label{defhatmun}
\hat{\mu}^N(t) = \mu^N ( t+\tau_N) ,\qquad t \geq0.
\end{equation}
Also let
%
%
\begin{equation}
\label{defhatn}
\hat{n}^N(t) = N \hat{h}^N(t) = n^N ( t+\tau_N),\qquad t \geq0.
\end{equation}
Let $\hat{\pi}^N_0 \in\mathcal{P}(\mathcal{M}^N_a(E \times[0,1]))$ be
the distribution of $\mu^N(\tau_N)=\hat{\mu}^N(0)$ and define $\pi^N
\in\mathcal{P}(S^N)$ by $\pi^N=\int_{S^N_0} \alpha(y,\cdot) \hat{\pi
}^N_0(dy)$. Let $X^N$ be the unique solution to the martingale problem
for $(A^N,\pi^N)$. Note that $X^N$ lives in the space
$S^N=\bigcup_{n=0}^{N} ( E \times[0,1] )^n$ and for any
$t\geq0$, $|X^N(t)|=\hat{n}^N(t)$. The process~$\hat{h}^N$ converges
to $h_{\mathrm{eq}}$ uniformly over compact time intervals [from part~(A) of Theorem \ref{mainconvthm}]. Hence, $\hat{n}^N$ converges to
$\infty$ uniformly over compact time intervals as well.

For part (B) of Theorem \ref{mainconvthm} we assume that the sequence
of random variables $\{\hat{\mu}^N(0)\}$ converges in distribution to
$\mu(0)$ as $N \to\infty$. Let $\hat{\pi}_0 \in\mathcal{P} (
\mathcal{M}_1(E \times[0,1]) )$ be the distribution of $\mu(0)$
and $\pi_0 \in\mathcal{P} ( \mathcal{P}(E \times[0,1]) )$
be the distribution of $\mu(0)/h_{\mathrm{eq}}$. Our assumption implies
that $\hat{\pi}^N_0$ converges weakly to $\hat{\pi}_0$. Due to part (A)
of Theorem \ref{mainconvthm}, this is equivalent to saying that the
distributions of $\mu^N(\tau_N)/h^N(\tau_N)$ converge weakly to $\pi_0$.

Now sample a probability measure $\mu$ from $\pi_0$ and let $
(Y_1,Y_2,\ldots)$ be an infinite sequence of exchangeable random
variables with de Finetti measure~$\mu$. Let $\pi\in\mathcal{P}((E
\times[0,1])^{\infty})$ be the corresponding distribution of
$(Y_1,Y_2,\ldots)$. Since \mbox{$\hat{\pi}^N_0 \Rightarrow\hat{\pi}_0$}, we
also have $\pi^N \Rightarrow\pi$.

From now on consider $X^N$ as a process over $(E \times[0,1])^{\infty
}$ in which the components greater than $N$ do not vary. The space $(E
\times[0,1])^{\infty}$ is given the usual product topology.

We can regard any function $f \in\mathcal{C} \cap B((E \times
[0,1])^{m})$ as a function over $(E \times[0,1])^{\infty}$ by defining
it as $f(x)=f(x^{|m})=f(x_1,\ldots,x_m)$ for any $x \in(E \times
[0,1])^{\infty}$. By the definition of $X^N$, for any $f \in\mathcal
{C} \cap B((E \times[0,1])^{m})$,
%
%
\begin{equation}
\label{mgrelcompxn}
M^N_{X,f}(t) = f(X^N(t))-\int_{0}^{t} A^N f(X^N(s))\,ds
\end{equation}
is a martingale. Define another process
%
%
\begin{equation}
\label{defhatzn}
\hat{Z}_N (t) = \gamma(X^N(t)) = \frac{1}{N} \sum_{i=1}^{\hat{n}^N(t)}
\delta_{X^N_i(t)},\qquad t\geq0.
\end{equation}
From Lemma \ref{usingmm}, the process $\hat{Z}_N$ has the same
distribution as the process~$\hat{\mu}^N$. Hence, if $F \in\bar
{\mathcal{C}}$ is given by $F(\mu)= \langle f,\mu^{(m)}
\rangle$, then
%
%
\begin{equation}
\label{mgrelcompzn}
M^N_{ \hat{Z},F} = F(\hat{Z}^N (t))-\int_{0}^{t} \mathbb{A}^N F(\hat
{Z}^N (s))\,ds
\end{equation}
is also a martingale. If $|X^N(t)|=\hat{n}^N(t)>m$, then
\begin{eqnarray*}
&&
A^N f(X^N(t))\\
&&\qquad = \frac{D}{2} \sum_{i=1}^m \Delta_{i} f(x)
\\
&&\qquad\quad{} + 2 k_{\mathrm{fb}} \biggl( \frac{N-\hat{n}^N(t)}{\hat{n}^N(t)+1}
\biggr) \sum_{1 \leq i < j \leq m} \bigl(f( \theta_{ij} (X^N(t)))-f(x)\bigr) \\
&&\qquad\quad{} + k_{\mathrm{on}} \biggl( \frac{N-\hat{n}^N(t)}{\hat{n}^N(t)+1}
\biggr) \sum_{i=1}^{m} \int_{E} \int_{0}^{1} \bigl( f( \theta_{i}
(X^N(t) ,(y,r) ) )-f(x)\bigr) \vartheta(dy)\,dr \nonumber.
\end{eqnarray*}
The \textit{death} term drops out because only the molecule at the
highest level is allowed to die and $f$ depends on only the first $m$
levels. From above it can be easily seen that for any positive integer
$m$ and $f \in\mathcal{C} \cap B((E \times[0,1])^{m})$, the supremum
of the process $A^N f(X^N(\cdot))$ over compact time intervals stays
bounded as $N \to\infty$. Using this fact along with (\ref
{mgrelcompxn}), (\ref{mgrelcompzn}) and (\ref{mmrelation}), it is easy
to argue that the sequence of processes $\{(X^N,\hat{Z}_N)\}$ is
relatively compact in $D_{(E \times[0,1])^{\infty}\times\mathcal
{M}_{1}(E \times[0,1])}[0,\infty)$ (see Corollary 9.3 and Theorem 9.4
in Ethier and Kurtz \cite{EK}). Suppose $(X,\hat{Z})$ is any limit
point and $(X^N,\hat{Z}_N) \Rightarrow(X,\hat{Z})$ along the
subsequence $k_N$. By the continuous mapping theorem and the
boundedness of $f\in\mathcal{C} \cap B((E \times[0,1])^{m})$, the
sequence of martingales $\{M^N_{X,f}(t) \}$ converges along the
subsequence $k_N$ to
%
%
\begin{equation}
\label{martmx}
M_{X,f}(t) = f(X(t))-\int_{0}^{t} A_m f(X(s))\,ds,
\end{equation}
which is a martingale with respect to the filtration generated by $X$.
The operator $A_m$ is given by
%
%
\begin{eqnarray}
\label{limAm}\quad
A_m f(x) & = & \frac{D}{2} \sum_{i=1}^m \Delta_{i} f(x) + 2 k_{\mathrm
{fb}} \biggl( \frac{1-h_{\mathrm{eq}}}{h_{\mathrm{eq}}} \biggr) \sum_{1
\leq i < j \leq m} \bigl(f( \theta_{ij} (x))-f(x)\bigr)
\nonumber\\[-8pt]\\[-8pt]
&&{}+ k_{\mathrm{on}} \biggl( \frac{1-h_{\mathrm{eq}}}{h_{\mathrm{eq}}}
\biggr) \sum_{i=1}^{m} \int_{E} \int_{0}^{1} \bigl( f( \theta
_{i} (x ,(y,r) ) )-f(x)\bigr) \vartheta(dy)\,dr \nonumber
\end{eqnarray}
for any $f \in\mathcal{D }(A_m)=\mathcal{C} \cap B((E
\times[0,1])^{m})$. The operator $A_m$ is the generator for the
process determined by the first $m$ levels of the limiting process $X$.
We can easily check that the martingale problem for $A_m$ is well posed
due to the same reasons that were given for $A^N$. Taking $\mathcal
{D}(A) = \bigcup_{m=1}^{\infty}\mathcal{D}(A_m)$ and defining
$A f = A_m f$ if $f \in\mathcal{D }(A_m)$, we see that the
martingale problem for $A$ is well posed. The distribution of $X^N(0)$
(denoted by $\pi^N$) converges to $\pi$. From (\ref{martmx}) we know
that for any positive integer $m$, the process followed by the first
$m$ levels of $X$ has generator~$A_m$. Hence, $X$ is the unique
solution to the martingale problem corresponding to $(A,\pi)$.

Let
\[
\gamma_N = \inf\{ t \geq0 \dvtx\hat{n}^N(t) =0 \}
\]
and for any $0 \leq t < \gamma_N$ define
%
%
\begin{equation}
\label{definettixn}
Z_N(t) = \frac{1}{\hat{n}^N(t)} \sum_{k=1}^{\hat{n}^N(t)} \delta_{X^N_k(t)}.
\end{equation}
Observe that
\[
\hat{Z}_N(t) = \frac{1}{N} \sum_{k=1}^{\hat{n}^N(t)} \delta_{X^N_k(t)}
=\biggl( \frac{\hat{n}^N(t)}{N} \biggr) \Biggl(\frac{1}{\hat{n}^N(t)} \sum
_{k=1}^{\hat{n}^N(t)} \delta_{X^N_k(t)}\Biggr) = \hat{h}^N(t) Z_N(t).
\]
The process $\hat{h}^N$ converges to the constant process $h_{\mathrm
{eq}}$. Therefore, $\gamma_N \to\infty$ in probability and for any
$t\geq0$, $\hat{n}^N(t) \to\infty$ in probability. Define $Z$ to be
the process
\[
Z(t) = \frac{\hat{Z}(t)}{h_{\mathrm{eq}}},\qquad t\geq0.
\]
Then $Z(0)$ has distribution $\pi_0$ and since $(X^N,\hat{Z}_N)
\Rightarrow(X,\hat{Z})$ along the subsequence $k_N$, we must have that
$(X^N,Z_N) \Rightarrow(X,Z)$ along the same subsequence. Notice that
for any $t$, $Z_N(t) \Rightarrow Z(t)$ implies that $Z_N^{(m)}(t)
\Rightarrow Z^{m}(t)$.
From the exchangeability of $X^N(t)$ we get that for any $f \in\mathcal
{C} \cap B((E \times[0,1])^{m})$,
\[
E ( f(X^N_1(t),\ldots,X^N_m(t)) ) = E \bigl( \bigl\langle f,
\hat{Z}_N^{(m)}(t) \bigr\rangle\bigr) = E \bigl( \bigl\langle f,
Z_N^{(m)}(t) \bigr\rangle\bigr)
\]
for any $0 \leq t < \gamma_N$.
Passing to the limit along the subsequence $k_N$, we obtain
\[
E ( f(X_1(t),\ldots,X_m(t)) ) = E ( \langle f,
Z^{m}(t) \rangle)
\]
for any $t \geq0$. It shows that conditional on $Z(t)$, $X_1(t),
X_2(t), \ldots$ are i.i.d. random variables with distribution $Z(t)$.
Hence, for any $t \geq0$, $X(t)$ is exchangeable with de Finetti
measure $Z(t)$. Therefore,
%
%
\begin{equation}
\label{definettix}
Z(t) = \lim_{n \to\infty} \frac{1}{n} \sum_{i=1}^{n}
\delta_{X_i(t)}\qquad
\mbox{a.s.}
\end{equation}
and if $\{\mathcal{F}^Z_t\}$ is the filtration generated by the process
$Z$, then
\[
E ( f(X_1(t),\ldots,X_m(t)) | \mathcal{F}^Z_t ) =
\langle f, Z^{m}(t) \rangle.
\]
Conditioning (\ref{martmx}) with respect to $\{\mathcal{F}^Z_t\}$, we
obtain that
\[
\langle f, Z^{m}(t) \rangle- \int_{0}^{t} \langle A_m
f,Z^{m}(s) \rangle \,ds
\]
is a martingale. If we define an operator $\mathbb{A}$ by
\[
\mathbb{A} F (\mu) = \langle A_m f,\mu^{m}\rangle
\]
for $F(\mu) = \langle f, \mu^{m} \rangle$, then this
definition agrees with the definition of the Fleming--Viot generator
$\mathbb{A}$ given in Section \ref{sec2} by (\ref{gena}). The
martingale problem
for $\mathbb{A}$ is well posed and, hence, $Z$ is the unique solution
to the martingale problem for $(\mathbb{A},\pi_0)$.

From the discussion above it is clear that $(X^N,\hat{Z}_N) \Rightarrow
(X,\hat{Z})$, where $\hat{Z} = h_{\mathrm{eq}} Z$ and $Z$ is a
Fleming--Viot process with generator $\mathbb{A}$. Since the process~%
$\hat{\mu}^N$ has the same distribution as the process $\hat{Z}_N$, we
also have $\hat{\mu}^N \Rightarrow h_{\mathrm{eq}} Z$ in $D_{\mathcal
{M}_1 (E \times[0,1] )}[0,\infty)$.
This proves part (B) of Theorem \ref{mainconvthm} with the process $\nu
$ being the same as the process $Z$.
\begin{remark}
\label{convergenceofinitialdistribution}
The approach of particle construction that we used to prove part (B) of
Theorem \ref{mainconvthm} can also be used to show that the
distributions of $\hat{\mu}^N(0) = \mu^N(\tau_N)$ converge along the
entire sequence. Since $\tau_N \to0$ a.s., we first do a random time
change $\gamma^N$ (which is a bijection from $[0,\infty)$ to $[0,\infty
)$ depending on the population size $n^N$) such that $\gamma^N(\tau
_N) \to\rho$ a.s., where $\rho$ is a positive random variable.
We then use the particle construction $\hat{X}^N(t) = (\hat
{X}^N_1(t),\hat{X}^N_2(t),\ldots)$ similar to the one used here,
except that the birth, death and immigration rates are altered
according to the random time change. Next we show that the process $\hat
{X}^N$ converges on the random time interval $[0,\gamma^N(\tau
_N)]$ as $N \to\infty$. With a bit more work it is possible to
conclude from this convergence that the distributions of $\mu^N(\tau
_N)$ converge as well.
\end{remark}

\subsection{\texorpdfstring{Proof of Theorem \protect\ref{theoremspread}}
{Proof of Theorem 2.7}}
The membrane molecules are doing speed~$D$ Brownian motion on the sphere of radius
$R$, which we call $E$. Suppose the sphere $E$ is embedded in $\mathbb
{R}^3$ with its center at the origin. Let $B=(B_1,B_2,B_3)^{T}$ be a
Brownian motion on $E$ with speed $D$ and let $W=(W_1,W_2,W_3)^{T}$ be
a standard Brownian motion in $\mathbb{R}^3$. Henceforth, let $\langle
\cdot,\cdot\rangle$ denote the standard inner product in $\mathbb{R}^3$
and let \mbox{$\| \cdot\|$} denote the corresponding Euclidean
norm. From Stroock \cite{St71} it follows that we can express $B$ as
the solution of It\^{o}'s equation
%
%
\begin{equation}
\label{itosbm0}
dB=\sqrt{D}\biggl(I-\frac{B B^{T}}{ R^2 }\biggr)\,dW-D \frac{B}{R^2}\,dt.
\end{equation}
From above, it is immediate that for any $t \geq0$,
%
%
\begin{equation}
\label{exb3}
E(B_i(t))= B_i(0) e^{-{2D}t/{R^2}  } \qquad\mbox{for } i=1,2,3.
\end{equation}

\begin{lemma}
\label{surfsp}
Let $B$ and $\bar{B}$ be two independent speed $D$ Brownian motions on
the sphere $E$. Then for any $t>0$,
\[
E\bigl( \| B(t)-\bar{B}(t) \|^2\bigr)= 2R^2\biggl( 1- \frac
{\langle B(0),\bar{B}(0)\rangle}{R^2} e^{-{2D}t/{R^2}  }\biggr).
\]
\end{lemma}
\begin{pf} This result is a consequence of the simple calculation below:
\begin{eqnarray*}
&&E \bigl(\| B(t)-\bar{B}(t) \|^2\bigr)\\
&&\qquad= E \bigl(
\bigl(B_1(t)-\bar{B}_1(t)\bigr)^2+\bigl(B_2(t)-\bar{B}_2(t)\bigr)^2
+\bigl(B_3(t)-\bar{B}_3(t)\bigr)^2
\bigr) \\
&&\qquad = E \bigl( B_1^2(t)+B_2^2(t)+B_3(t)^2+\bar{B}_1^2(t)+\bar
{B}_2^2(t)+\bar{B}_3(t)^2 \\
&&\qquad\quad\hspace*{36pt}{}-2 B_1(t)\bar{B}_1(t)-2 B_2(t)\bar{B}_2(t)-2 B_3(t)\bar{B}_3(t)
\bigr) \\
&&\qquad = 2R^2-2 E(B_1(t))E(\bar{B}_1(t))-2 E(B_2(t))E(\bar{B}_2(t))-2
E(B_3(t))E(\bar{B}_3(t)) \\
&&\qquad = 2R^2\biggl( 1- \frac{\langle B(0),\bar{B}(0)\rangle}{R^2}
e^{-
{2D}t/{R^2}  }\biggr) \qquad[\mbox{using (\ref{exb3})}].
\end{eqnarray*}
\upqed\end{pf}

Recall the definition of $S_p$ given by (\ref{spreadst}) and the
definition of the process~$X$. We assume that we are at stationarity
and, hence, we can also assume that $X$ is defined for all $t \in
(-\infty,\infty)$. At any fixed time $t$ the sequence
$X(t)=(X_1(t),X_2(t),\ldots)$ is exchangeable and its de Finetti
measure $Z(t)$ has the same distribution as $\nu(t)$. Thus, the
distribution of two molecules sampled from $\nu(t)$ is the same as the
distribution of the first $2$ levels $X_1(t)$ and~$X_2(t)$. For $i=1,2$
let $X_i(t)=(Y_i(t),C_i(t))$, where $Y_i(t) \in E$ and $C_i(t) \in
[0,1]$. From the calculation in Section \ref{sec2} we can write
%
%
\begin{equation}
\label{spread1}
S_p = E\bigl( \|Y_1(t)-Y_2(t)\|^2 | C_1(t)=C_2(t) \bigr).
\end{equation}
The process determined by the first two levels of $X$ evolves according
to the generator $A_2$ given by (\ref{limAm}) with $m=2$. From the
definition of $A_2$ it is clear that level $2$ \textit{looks down} to
level $1$ at rate $2k_{\mathrm{fb}} \alpha$, where $\alpha=
(1-h_{\mathrm{eq}})/h_{\mathrm{eq}}$. Moreover, at both the levels
there is an \textit{immigration} event at rate $k_{\mathrm{on}}\alpha$,
in which a molecule with a uniformly chosen type in $E \times[0,1]$ is
inserted at that level. In between these lookdowns and immigrations,
the molecules at levels $1$ and $2$ are diffusing on the membrane
according to independent speed $D$ Brownian motions.

The quantity $S_p$ can be calculated by tracing back the history from
time~$t$. Let $\tau_{12}$ be the last lookdown time between the first
two levels and $\tau_i$ be the last immigration time at level $i$ for
$i=1,2$. The random variables~$\tau_{12}$,~$\tau_1$ and $\tau_2$ are
independent and exponentially distributed with rates $2k_{\mathrm{fb}}
\alpha$, $k_{\mathrm{on}}\alpha$ and $k_{\mathrm{on}}\alpha$,
respectively. Let $\tau$ be the minimum of $\tau_1,\tau_2$ and $\tau
_{12}$ and so it is an exponential random variable with rate
$2(k_{\mathrm{on}}+k_{\mathrm{fb}})\alpha$.

The molecules at levels $1$ and $2$ will be in the same clan provided
$\tau=\tau_{12}$.
Molecules at levels $1$ and $2$ were at the same place at time $t-\tau$
and have been doing independent speed $D$ Brownian motions on the
sphere $E$ since then. Using Lemma \ref{surfsp}, we get
\begin{eqnarray*}
&&E \bigl( \| Y_1(t)-Y_2(t) \|^2 | C_1(t)=C_2(t) \bigr) \\
&&\qquad= 4R^2 (k_{\mathrm{on}}+k_{\mathrm{fb}})\alpha\int_{0}^{\infty}
( 1- e^{-{2D}s/{R^2}  }) e^{-2(k_{\mathrm{on}}+k_{\mathrm
{fb}})\alpha s}\,ds \\
&&\qquad = \frac{2 D}{( (k_{\mathrm{on}}+k_{\mathrm{fb}})\alpha+
{D}/{R^2} )}.
\end{eqnarray*}
This proves Theorem \ref{theoremspread}.

\section*{Acknowledgments}

I wish to sincerely thank my adviser, Professor Tho\-mas G. Kurtz, for
his constant support and guidance. A very special thanks to Professor
Sigurd Angenent for introducing me to this problem and asking many
interesting questions. I also wish to thank Professor Steve Altschuler
and Professor Lani Wu for inviting me to their lab at the University of
Texas, Southwestern and giving me the opportunity to better understand
the biological aspects of this problem.


%

%
\printaddresses

\end{document}